\documentclass[reqno]{amsart}

\usepackage{amscd}
\usepackage{amssymb}
\usepackage{amsmath}
\usepackage{mathrsfs}
\usepackage[dvips]{graphicx}
\usepackage{amsfonts}
\usepackage{amsopn}
\usepackage{cases, comment}
\usepackage{hyperref}

\numberwithin{equation}{section}
\newcounter{mtheorem}

\newtheorem{thm}{Theorem}[section]

\newtheorem{proposition}[thm]{Proposition}

\newtheorem{lemma}[thm]{Lemma}
\newtheorem{definition}[thm]{Definition}

\newtheorem{remark}{Remark}

\def\Reals{\mathbb{R}}

\def\R2d{{\Reals}^{2d}}

\def\calA{{\mathcal A}}

\def\calD{\mathcal D}
\def\calE{\mathcal E} 
\def\calF{\mathcal F}
\def\calG{\mathcal G}
\def\calH{{\mathcal H}}
\def\calJ{{\mathcal J}}
\def\calI{{\mathcal I}}
\def\calK{{\mathcal K}}
\def\calL{{\mathcal L}}
\def\calM{{\mathcal M}}
\def\calPac_1on{{\mathcal M}on}

\def\calR{{\mathcal R}}
\def\calS{{\mathcal S}}
\def\calT{{\mathcal T}}
\def\calU{{\mathcal U}}
\def\calV{{\mathcal V}}
\def\calW{{\mathcal W}}

\def\calPac{{\mathcal P}^{ac}}

\def\0{{\bf 0}}

\def\p{{\bf p}}
\def\q{{\bf q}}

\def\x{{\bf x}}

\def\w{{\bf w}}

\def\id{{\bf id}}

\def\div{{\rm div}}

\def\endproof{\ \hfill $\square$ \bigskip} \def\proof{\noindent {\bf Proof:}}
\def\proof#1{\noindent {\bf Proof{#1}:}}

\def\calPac{{\mathcal P}^{ac}}

\def\p{{\bf p}}
\def\q{{\bf q}}

\def\x{{\bf x}}

\def\w{{\bf w}}

\def\id{{\bf id}}

\def\div{{\rm div}}
 
\def\spt{{\rm spt}}

\title{ On a model of axisymmetric flows in a free boundary domain}

\author{Marc Sedjro}
\thanks{ African Institute for Mathematical Sciences  Tanzania, ({\tt sedjro@aims.ac.tz})}
  
\begin{document}

\maketitle

\begin{abstract}
We consider a model of axisymmetric flows for a free boundary  vortex  embedded in a statically stable fluid at rest. We identify  the boundary of the vortex by solving a variational problem. Then, we reduce the analysis of the dynamics of the vortex to the study of a class of continuity equations for which we construct a solution. 

\end{abstract}



\section{Introduction}
\label{}
Axisymmetric flows are appropriate models to describe idealized tropical cyclones. Typically, they describe the evolution of a balanced vortex under the forcing effects  of tangential momentum and heat sources.  Though widely studied, these flows still present challenges as attested by recent studies and results; see \cite{HuangMont 2012}, \cite{BellMont2012}. In \cite{Craig},  Craig  derived a system of equations for flows that are almost circular in gradient balance, the so-called \textit{the almost axisymmetric flows}.  These flows model, in the absence of viscosity, the motion of a vortex in a rotating reference  frame where the  coriolis coeffficient is $\Omega>0$ and the gravity of earth is $g$. The vortex evolves at a velocity ${\bf u}=(u,v,w)$ in a domain where the potential temperature is $\theta$ and the pressure is $\varphi$, and is kept in an ambient fluid  at a prescribed temperature $\theta_0$. Under the effects of forcing terms $F_0$ and $F_1$, the equations for the $3-$dimensional axisymmetric flows are given in cylindrical coordinates with standard variables $(\lambda,r,z)$ by:  
\begin{equation}{\label{eq:craig}}
\begin{cases}
\frac{{\rm D} u}{{\rm D} t}+\frac{uv}{r}+2\Omega v +\frac{1}{r}\frac{\partial\varphi}{\partial \lambda}&=\frac{1}{r}F_0,\\
\frac{{\rm D} \theta}{{\rm D} t}& =F_1,\\
-\frac{u^2}{r}-2\Omega u +\frac{\partial\varphi}{\partial r}&=0, \\
-g\frac{\theta}{\theta_0}+\frac{\partial\varphi}{\partial z} &=0, \\
\frac{\partial}{\partial \lambda}(u)+ \frac{\partial}{\partial r}(rv)+\frac{\partial rw}{\partial z} &=0.
\end{cases}
\end{equation}

In the absence of forcing terms, the almost axisymmetric flows approximate the hydrostatic Boussinesq equations. Though a simpler model, almost axisymmetric flows still present some challenging regularity issues (see \cite{Sedjro}). In \cite{CulSed2014}, the authors introduced  two-dimensional flows derived from  \eqref{eq:craig}. These flows provide axisymmetric solutions to \eqref{eq:craig} and share the same stability states as the almost axisymmetric flows  as flow parcels follow displacements preserving angular momentum and potential temperatures (see \cite{EliassenK}). Building on the work in \cite{EliassenK}, \cite{Fjortoft} and \cite{Shutts}, they developed  a procedure  that uses the theory of optimal mass transport  to construct a solution to the two-dimensional flows within a moving  domain $\Gamma_{\varsigma_t}$ defined by
\begin{equation}\label{eq: boundary cond 1}
\Gamma_{\varsigma_t}:=\left\lbrace ( r,z) : \;\; 0\leq z\leq H,\;\;  r_0\leq r\leq \varsigma(t,z) \right\rbrace.
\end{equation}
 One key assumption in the model considered in \cite{CulSed2014}, is that the ambient temperature $\theta_0$ is constant. This assumption makes the problem tractable. However, when $\theta_0$ is constant, the ambient fluid loses its static stability.  In this paper, we consider the more physically relevant model in which the ambient temperature varies in function of the height level of the vortex.
 
\subsection{ Axisymmetric flows with forcing terms}
The axisymmetric flows are derived from \eqref{eq:craig} by assuming that  the  quantities $u,v,w, \varphi,\theta$ and the operator ${D  \over Dt}= {\partial \over \partial t}+u{\partial \over \partial r}+  v{\partial \over \partial r}+  w{\partial \over \partial z}$  are all independent of the angular variable $\lambda$. These considerations lead to the following system:

\begin{subnumcases}{\label{eq:non physical}}
\frac{{\rm D}  u}{{\rm D} t}+\frac{ u v}{r}+2\Omega  v &$=\frac{1}{r} F_0(t,r,z),\label{first1}$\\
\frac{{\rm D} \theta}{{\rm D} t} &$= F_1(t,r,z),\label{second1}$\\
\frac{ u^2}{r}+2\Omega -\frac{\partial\varphi}{\partial r}  &$=0,\label{third1}$ \\
g\frac{\theta}{\theta_0}-\frac{\partial \varphi}{\partial z} &=$0, \label{fourth1}$\\
 \frac{1}{r}\frac{\partial}{\partial r}(r v)+\frac{\partial  w}{\partial z} &=$0$.\label{fifth1}  
\end{subnumcases}

The equations \eqref{eq:non physical}  are to be solved in the domain $\Gamma_{\varsigma_t}$. Thus, we supplement \eqref{eq:non physical}  with a Neumann boundary condition on the rigid boundary $\Lambda _{rig}$ composed of sets  $\{r=r_0\}$, $\{z=0\}$, and $\{z=H\} $ and a kinematic boundary condition on the free boundary $\Lambda _{\varsigma}$ representing $\{ r=\varsigma(t,z)\}$  : 
\begin{equation}\label{eq: non physical free boundary}
\begin{cases}
 \langle(  v_t, w_t) , {\bf n_t}\rangle=0 \qquad \qquad \text{ on } \Lambda_{rig} ,\\
 \frac{\partial \varsigma_t}{\partial t} + w\frac{\partial \varsigma_t}{\partial z} \;\;\;=  v \qquad \qquad \text{on}\;\; \Lambda _{\varsigma}.
\end{cases}
\end{equation} 
Here ${\bf n_t}$  is the unit outward normal vector field at time $t$. On the free boundary, we impose the following condition on the pressure :
\begin{equation}\label{eq: bdary cond on pressure}
 \varphi(t,\varsigma(t, z),z)=0\qquad \text{for}\quad \varsigma(t, z)>0.
\end{equation}

  The Hamiltonian relevant to the system  (\ref{eq:non physical}) is given by
\begin{equation}\label{eq: Hamiltonian}
 \int_{\Gamma_{\varsigma}}\left( \dfrac
{u^2}{2}-\frac{g\theta}{\theta_0}z\right)rdrdz.
\end{equation}
From a meteorological point of view, we are looking for solutions for which the vortex is stable with respect to perturbations. To that aim, a stability condition is imposed on the pressure. Notably,

\begin{equation}\label{eq: stability condition}
\mathscr{A}(\theta_0(z))\nabla_{r,z}\left( \varphi + \Omega^2\frac{r^2}{2} \right)\quad \text{  is    \textit{invertible}.}
\end{equation}
Here, $\mathscr{A}(m)$ denotes the $2\times 2$ diagonal matrix $\text{Diag}(1,m)$.
\subsection{Hamiltonian and stable solutions}
We discover that the Hamiltonian in (\ref{eq: Hamiltonian}) plays an important role in the construction of solution to stable  axisymmtric flows. By making  the change of coordinate system $\Upsilon=(ru+r^2\Omega r^2)^2$ and $Z=g\theta$, this Hamiltonian can be written  solely in terms of a measure $\sigma$ provided that the stability condition \eqref{eq: stability condition} is satisfied. Subsequently, the Hamiltonian takes the form of the following functional :
  
  \begin{equation}\label{eq intro:primal1} 
\mathscr{P}(\mathbb{R}^2) \ni \sigma\longmapsto \calH(\sigma):= h(\sigma)+
  \inf_{\varrho\in\calR} W^2_2(\sigma, {\bf f}{\#}\mu_{\varrho})+ \int_{\mathbb{R}^2} E(\p) \mu_{\varrho}(d\p)  
  \end{equation}
  
  Here, $W_2$ denotes the $2-$Wasserstein distance, ${\bf f}(s,z)=(s, z/\theta_0(z))$, $h(\sigma)= \int_{\mathbb{R}^2} \left( {\Upsilon \over 2 r_0^2}-\Omega \sqrt \Upsilon \right) \sigma(d\q)$
with  $\q=(\Upsilon,Z)$, and the function $E$ is defined by 
$$ E(\p)=f_0(s)-\dfrac{1}{2}\left( s^2+z^2/\theta_0^2(z)\right) \quad \text{with}\quad \p=(s,z)\quad \text{and}\quad f_0(s)={r_0^2 \Omega^2 \over 2(1-2 r_0^2 s)}. $$
 $\calR$ denotes the set of the Borel functions  $\varrho : [0,H]\longmapsto [0,1/(2r_0^2)$.  For any map $\varrho\in\calR$, we associate the Borel measure $\mu_{\varrho}$ which is absolutely continuous with respect to the Lebesgue measure and  whose density is given by $\left( \frac{2f_0(s)}{\Omega^2}\right)^2 \chi_{D_{\varrho}}(s,z)$ with 
$D_{\varrho}= \left\lbrace (s,z): 0\leq s\leq \varrho(z), z\in [0,H] \right\rbrace$. While the $W_2(\sigma,\cdot)$ has some convexity properties, the function $E$ is not convex. To study the existence and the uniqueness of a minimizer in (\ref{eq intro:primal1}), we consider a dual formulation :

  \begin{equation}\label{eq intro:dual1} 
    \calH(\sigma)= h(\sigma)+\sup \left(  \int_{\mathbb{R}^2} - \Psi(\q)\; \sigma(d\q)+\inf_{\varrho \in \calR} \int_0^H \int_0^{\varrho(z)} \left( f_0(s)-P(\p,m) \right)  \delta_{\theta_0(z)}(dm)\mu_{\varrho}(d\p)\right). 
  \end{equation}

  The supremum in (\ref{eq intro:dual1}) is taken over the set 
  \begin{equation}\label{eq intro:for calU}
   \mathcal{U}:=\Bigl\{ (P, \Psi)\in C( \mathcal{W}\times \calI_0)\times C(\Reals^2_+) : P(\p,m)+\Psi(\q) \geq  c(\p,m, \q)  \text{ for all } (\p,\q)\in \mathcal{W}\times\mathbb{R}_+^2,\; m\in\calI_0 \Bigr\}
  \end{equation}
  where the cost function $c$ is given by 
 $$c(\p,m,\q)= s\Upsilon+\dfrac{zZ}{m}\qquad\hbox{with}\quad \p=(s,z)\quad  \hbox{and}\quad\q=(\Upsilon,Z).$$
 Assume that $ \varrho^{\sigma}_0$ is a minimizer in (\ref{eq intro:primal1}) and that $ (P^{\sigma},\Psi^{\sigma})$ are c-transform of each other in the sense of definition \ref{def: c-transf}  and  are maximizers of \eqref{eq intro:dual1}. Then, $ \varrho^{\sigma}_0$ is uniquely determined and the maps $\mathcal{T}[P^{\sigma}]$ and $ \mathcal{S}[\Psi^{\sigma}]$
   defined respectively by 
 \begin{equation}\label{eq: subgradient P}
 \mathcal{T}[P^{\sigma}](\p)=\left( \partial_{s}P^{\sigma}(\p,\theta_0(z)), \;\theta_0(z)\partial_{z}P^{\sigma}(\p,\theta_0(z))\right) \qquad \p=(s,z)
 \end{equation}
 and 
 \begin{equation}\label{eq: subgradient Psi}
 \mathcal{S}[\Psi^{\sigma}](\q)=\left(\partial_{\Upsilon}\Psi^{\sigma}(\q), \;\phi^{-1}(\partial_{Z}\Psi^{\sigma}) (\q)\right)\qquad\quad\qquad \phi(z)= z/\theta_0(z)
 \end{equation}
  are essentially injective functions and we have $ \mathcal{S}[\Psi^{\sigma}]\circ\mathcal{T}[P^{\sigma}]= \id \quad  \mu_{\varrho_0^{\sigma}}$-a.e. Furthermore, we have that $\mathcal{T}[P^{\sigma}]=\mathscr{A}(\theta_0(z))\nabla P^{\sigma}$ pushes  forward $\mu_{\varrho_0^{\sigma}}$ onto $\sigma$ and $2(1-2r^2_{0} \varrho_0(z))P^{\sigma}(\varrho_0(z),z, \theta_0(z))=r^2_0\Omega^{2} \text{ on } \{\varrho_0^{\sigma}>0\}$. If, in addition, $\sigma$ is absolutely continuous with respect to Lebesgue then $\mathcal{T}[P^{\sigma}]\circ  \mathcal{S}[\Psi^{\sigma}]= \id \;\sigma-$a.e.

\subsection{Continuity equation corresponding to the 2D Axisymmetric Flows with Forcing Terms}\hfill

\label{sec: Change of variables into the Dual space  and Justification}
A class of continuity equations plays a determining role in the construction of solutions to the  axisymmetric flows \eqref{eq:non physical} satisfying the stability condition \eqref{eq: stability condition}. Assume that $\sigma \in AC_2\left( 0, T;\mathscr{P}\left(  \mathbb{R}^2\right) \right)  $ and that $t\longrightarrow \sigma_t$ satisfies 
 \begin{equation}\label{eq00: Continuity equation unphysical problem}
 \begin{cases}
  \frac{\partial\sigma}{\partial t}+ \div (\sigma V_t[\Psi])=0, \qquad\qquad \mathcal{D}'\left((0,T)\times  \mathbb{R}^2\right) \\
 \sigma|_{t=0}=\sigma_0,
 \end{cases}
 \end{equation}
 with \begin{equation}\label{eq1 unphysical1: velocity in dual space} 
  V_t[\Psi]= {\textstyle \left(\;2\sqrt{\Upsilon} F_0\left( t,\frac{1}{\Omega}\sqrt{2f_0(\frac{\partial \Psi}{\partial \Upsilon})},\phi^{-1}\left( \frac{\partial\Psi}{\partial Z}\right)  \right)  ,\; g F_1\left( t,\frac{1}{\Omega}\sqrt{2f_0(\frac{\partial \Psi}{\partial \Upsilon})},\phi^{-1}\left( \frac{\partial\Psi}{\partial Z}\right) \right)  \right) }\quad\hbox{and}\quad \phi(z)=z/\theta_0(z)
 \end{equation}
  where $\Psi_t$ is such that $(P_t,\; \Psi_t)$  $c-$transforms of each other for each $t\in [0,T]$. Assume that $\varrho_t$ is monotone  and that $(P_t,\; \Psi_t,\varrho_t)$  solve uniquely
 \begin{equation}
 \begin{cases}\label{eq: main problem}
 \mathcal{T}[P]\#\mu_{\varrho}=\sigma, \\
 \mathcal{S}[\Psi]\circ \mathcal{T}[P]=\id \;\mu_{\sigma_{\varrho}},\\
P(\varrho(z),z,\theta_0(z))={\Omega^2 r_0^2\over 2(1 -2r_0^2 \varrho( z)) } \quad \hbox{on} \quad \{\varrho>0\}.
\end{cases}
\end{equation}
  
 Then, given enough regularity, we can construct a solution $u,v,w,\theta,\varphi,\varsigma$ to (\ref{eq:non physical})-(\ref{eq: bdary cond on pressure}) and \eqref{eq: stability condition}. As shown in section \ref{sec: Derivation of the continuity equation}, through the change of variable ${\bf s}_{\theta_0}(r,z)=(s[r],z,\theta_0(z))$ with $ 2s= 2s[r]:= r_0^{-2}-r^{-2}$, the quantities $\theta$, $\varphi$ and $\varsigma$ are obtained by 
  
  \begin{equation}\label{eq : physical var 1}
\begin{cases}
\theta_t(r,z)=\theta_0(z)\partial_{z}P_t\circ{\bf s}_{\theta_0}(r,z)\\
 \varphi_t(r,z)=P_t\circ{\bf s}_{\theta_0}(r,z)-\Omega^2r^2/2\\
\varsigma_t(z)= \frac{1}{\Omega}\sqrt{2f_0(\varrho_t(z))}
\end{cases}
\end{equation}
and the velocity field $(u,v,w)$ is given by
\begin{equation}\label{eq : physical var 2}
\begin{cases}
u_t=1/r\sqrt{\partial_{s}P_t\circ{\bf s}_{\theta_0}}-r\Omega\\
v_t= \left(\partial_t \mathbb{S}_1\right)\circ\mathbb{T} +2\sqrt{\partial_{s}P_t\circ{\bf s}_{\theta_0}}F_{0t}\left( \partial_{\Upsilon}\mathbb{S}_1\right)\circ\mathbb{T}
+ gF_{1t}\left( \partial_Z\mathbb{S}_1\right)\circ\mathbb{T}\\
w_t=  \left(\partial_t\mathbb{S}_2\right)\circ\mathbb{T} +2\sqrt{\partial_{s}P_t\circ {\bf s}_{\theta_0}} F_{0t}\left(\partial_{\Upsilon} \mathbb{S}_2\right)\circ\mathbb{T}
+ gF_{1t}\left( \partial_{Z}\mathbb{S}_2\right)\circ\mathbb{T}
\end{cases}
\end{equation}

Here, $\mathbb{T}=\mathcal{T}[P^{\sigma}]\circ{\bf s}$ and $\mathbb{S}= {\bf r}\circ\mathcal{S}[\Psi^{\sigma}]$ with ${\bf s}(r,z)=(s[r],z)$ and its inverse ${\bf r}(s,z)= \left(\frac{1}{\Omega}\sqrt{2f_0(s)}, z\right)$.

\subsection{ Plan of the paper}

This paper is organized as follows: In section 2 we collect notation, definitions and key assumptions throughout the paper.
In section 3, we explain how solutions to the axisymmetric flows can be constructed via the study of a class of continuity equations with enough regularity. In section 4, we study a variational problem that determines the free boundary, its regularity and the velocity fields associated to the class of continuity equations considered. In  Section 5, we study the stability of the free boundaries and the velocity fields governing these continuity equations. In section 6,  we follow a discretization scheme developped in \cite{AmbrosioGangbo} to construct solution for (\ref{eq00: Continuity equation unphysical problem}) and (\ref{eq1 unphysical1: velocity in dual space}).
\section{Notation, Definitions and Assumptions}
\label{sec: notation}
Throughout this paper, we use the  following notation, definitions and assumptions:
\subsection{Notation and Definitions}

\begin{itemize}
 \item  $g, r_0$ and $H$ are positive constants and we set
  $$ \mathcal{W}:=[0,1/2r_0^2)\times [0,H]\quad\hbox{and}\quad\mathcal{W}_{\infty}:=[0,\infty)\times [0,H] .$$
  $\overset{o}{\calW}$ denotes the interior of $ \mathcal{W}$.
  \item $f_0$ denotes the function on $[0, 1/2(r_0^2))$ defined by $f_0(s)={r_0^2 \Omega^2 \over 2(1-2 r_0^2 s)}. $  
  \item $\calR$ denotes the set of the Borel functions  $\varrho : [0,H]\longmapsto [0,1/(2r_0^2)$.  For any map $\varrho\in\calR$, we associate the set $D_{\varrho}:= \left\lbrace (s,z): 0\leq s\leq \varrho(z), z\in [0,H] \right\rbrace$ and the Borel  measure  $\mu_{\varrho}$ which is absolutely continuous with respect to the Lebesgue measure in $\mathbb{R}^2$ and  whose density  is given by $\left( \frac{2f_0(s)}{\Omega^2}\right)^2 \chi_{D_{\varrho}}(s,z)$.
  
  \item $\calR_0$ is the subset of $\calR$ for which $\mu_{\varrho}$ is a Borel probability measure.
  \item  $\calV:=(0,\infty)\times(0,\infty)$ and $B_a$ is the ball in $\mathbb{R}^2$ centered at $(0,0)$ and of radius $a$.\\
   We denote by $B^+_a:=B_a\cap \calV.$  
  \item $\mathcal{I}_0$ is an open bounded interval of $\mathbb{R}_+$ such that $0\notin \bar{\mathcal{I}_0}$. Here, $ \bar{\mathcal{I}_0}$ denotes the closure of $\mathcal{I}_0$.
  
\item Let $D\in \mathbb{N}$. For any naturel numbers $1\leq i< j< k\leq D$, $\pi^i,\;\pi^{i,j}\;\pi^{i,j,k}$ denote the projection operators on $\mathbb{R}^D$ defined respectively by $\pi^i(x_1,\cdots x_D)=x_i,\quad\;\pi^{i,j}(x_1,\cdots x_D)=(x_i,x_j)$ and $\pi^{i,j,k}(x_1,\cdots x_D)=(x_i,x_j,x_k)$. $D$ will be determined by the context.
\item  $\calL^D$ denotes the Lebesgue measure in $\mathbb{R}^D$.

\item If $\sigma$ is a measure on $\mathbb{R}^D$ absolutely continuous with respect to $\calL^D$ then we denote by $\dfrac{d\sigma}{d\calL^D}$ the Radon-Nycodym derivative of $\sigma$ with respect to $\calL^D$.

\item  $\mathscr{P}(\mathbb{R}^D)$ is the set of all Borel probability measures on $\mathbb{R}^D$.

\item For $\mu\in\mathscr{P}(\mathbb{R}^D)$, we denote by $\spt(\mu)$ the support of $\mu$ defined by 
$$ \spt(\mu):=\left\lbrace \x\in \mathbb{R}^D : \mu(O)>0 \text{ for any open set } O \text{ of }\mathbb{R}^D\right\rbrace.$$
\item   $\mathscr{P}_p(\mathbb{R}^D)$ ($1\leq p<\infty$) denotes the set of  probability measures with finite $p-$ moments:
$$\int_{\mathbb{R}^D}|\x|^p d\mu(\x) <\infty. $$
 \item We denote by $\mathscr{P}^{ac}(\mathbb{R}^D)$  the set of all elements of  $\mathscr{P}(\mathbb{R}^D)$ that are absolutely continuous with respect to Lebesgue. 
\item  Let  $\mu \in \mathscr{P}(\mathbb{R}^D)$ and $\calT: \mathbb{R}^D\longrightarrow \mathbb{R}^D$ a Borel map. The push-forward of $\mu$ through $\calT$, denoted by $\calT\#\mu\in \mathscr{P}(\mathbb{R}^D)$, is defined by
   $$\calT\#\mu (A):=\mu(\calT^{-1}(A)) \text{ for any Borel set  }A\in\mathbb{R}^D. $$ 
\item Given  $\mu, \;\nu \in \mathscr{P}_p\left(\mathbb{R}^D \right) $,  the $p$-Wasserstein distance between  $\mu$ and $\nu$  is defined as
\[
    W^{p}_{p}(\mu,\nu)=     \min\left\lbrace \int_{\mathbb{R}^d\times\mathbb{R}^d}|x-y|^pd\gamma: \frak p^1\# \gamma=\mu \quad\text{and}\quad \frak p^2\#\gamma=\nu\right\rbrace, 
 \]
$\frak p^1$ and $\frak p^2$ denote respectively the first and second projections on $\mathbb{R}^D\times\mathbb{R}^D$.
 \item Let $\left(\calS, \mathbf{d} \right) $ be a complete metric space,  $a$ and $b$ be real numbers such that $a<b$. A curve $ \mathbf{s}:(a,b)\longrightarrow \calS$ is said to belong to  $AC_m\left(a,b;\; \calS \right)$  if there exists $ {\bf p}\in L^m(a,b)$ such that 
 $$\mathbf{d}(\mathbf{s}(t),\mathbf{s}(s))\leq \int_t^s {\bf p}(r)dr \qquad  \text{for all } a<t\leq s<b.$$
 
 Curves in  $AC_m\left( a,b; \calS\right)$  are said to be $m-$absolutely continuous, see \cite{AmbrosioGangbo}. 
 
\item For any matrix $A$, we denote by $A^{\tau}$ the transpose of $A$.
 \item For time-dependent functions, we use the notation $S_t(\cdot,\cdot)=S(t,\cdot,\cdot)$ for convenience.
  \item $\calA: \calI_0\longrightarrow\mathbb{R}^{2\times 2}$ defined by $m\longrightarrow\mathcal{A}(m)=\begin{bmatrix}1 &0\\0&m\end{bmatrix}$.
  \item We use the notation $\p=(s,z)$, $\q=(\Upsilon,Z)$ and denote by $c$ the cost function on $\mathbb{R}^2\times\calI_0\times\mathbb{R}^2$ defined  by 
  \begin{equation}\label{eq: cost function}
  c(\p,m,\q)= s\Upsilon+\frac{zZ}{m}.
  \end{equation}
\end{itemize} 
\subsection{Assumptions}\hfill
\begin{itemize}
 \item The function $\theta_0:[0,H]\longrightarrow\calI_0$ is assumed smooth on $(0,H)$ and satisfies the following conditions:
\begin{itemize}
\item[(A1)] $\theta_0(z)-z\theta_0'(z)>0$.
\item[(A1')] $\inf_{0\leq z\leq H}\left\lbrace \theta_0(z)-z\theta_0'(z)\right\rbrace >0$ for all $z\in(0,H)$. 
\item[(A2)]$\theta_0$  is Lipschitz continuous.
\end{itemize}
Typically, $\theta_0(z)=\left( A+Bz^{\alpha}\right)^{\beta} $ with $A>0,\; B>0$, $\beta\geq 1$ and $\alpha\beta\leq 1$ satisfies conditions (A1') and (A2). The condition (A1) implies that $\phi(z):=z/\theta_0(z)$ is strictly increasing.\\
\item We assume that $ F_0:= F_{0t}(r,z)$ and $  F_1= F_{1t}(r,z)$ are  such that  $F_0,  F_1\in C^{1}((0,\infty)\times\mathbb{R}^2)$ and satisfy the following conditions:

\begin{itemize}
 \item[(B1)]  $0\leq  F_0,\; g F_1 \leq M  $  for some positive constant $ M $.
 \item[(B2)]  $ \partial_z F_0=\partial_r F_1=0.$
 \item[(B3)] $\partial_r F_0,\; \partial_{z} F_1>0. $
\end{itemize}
\end{itemize}



                  \section{continuity equations and axisymmetric flows      }
\label{sec: Derivation of the continuity equation}
In this section, we discuss  how one can derive a solution for the axisymmetric flows from the study of a class of continuity equations. We point out that this derivation relies on the assumption that we have enough regularity for solutions to this class of continuity equations. Let $\textbf{v}=(v,w)$ be a smooth velocity field  and $\varsigma_t$ a smooth function  such 

\begin{equation}\label{eq1: conservation of mass physical} 
 \begin{cases}
  \frac{1}{r}\partial_{r}(rv)+\partial_{z} w=0 \quad\;\text{on} \;\; \Gamma_{\varsigma_t},\\
\langle{{\bf v}}_t , {\bf n}_t \rangle=0 \qquad\qquad \;\text{on}\;\;  \Lambda_{rig},\\
 \partial_t \varsigma_t + w\partial_z \varsigma_t=v   \qquad\text{on}\;\; \Lambda_{\varsigma_t}.\\
 \end{cases}
\end{equation}
Here, ${\bf n}_t$ is the outward unit normal vector to the rigid boundary $\Lambda_{rig}$ for each $t$ fixed. 
The following lemma is proved in dimension 3 in \cite{Sedjro}. We reproduce the proof in dimension 2 for the reader's convenience.
\begin{lemma}\label{le: conservation of total mass} 
Let $T>0$ and  $\sigma \in AC_2\left(0, T; \mathscr{P}\left(  \mathbb{R}^2\right) \right)$ and $V\in C^1\left( (0,T)\times \mathbb{R}^2\right) $ such that $$\frac{\partial\sigma}{\partial t}+ \div (\sigma V_t)=0, \quad \mathcal{D}'\left((0,T)\times  \mathbb{R}^2\right).$$
Let $\calG=\left(\calG^1, \calG^2 \right)  $ be a smooth  function on $(0, T)\times\mathbb{R}^2$  such that $\calG_t$ is invertible with inverse $\calF_t$ for each $t\in[0, T]$. Assume that there exists $\varsigma$ such that for $t\in(0,T)$ we have $\calG_t\#\sigma_t=(r\chi_{\Gamma_{\varsigma_t}}\calL^2)$. Define $v,w$ respectively by 
\begin{equation}\label{Def v,w}
v=\partial_t \calG^1_t\circ \calF +V_1\partial_{\Upsilon}\calG^1 \circ \calF +V_2\partial_{Z}\calG^1 \circ \calF\qquad w=\partial_t \calG^2_t\circ \calF +V_1\partial_{\Upsilon}\calG^2 \circ \calF +V_2\partial_{Z}\calG^2 \circ \calF
\end{equation}
 Then $v,w, \varsigma$ solve (\ref{eq1: conservation of mass physical}) 
\end{lemma}
\proof{}
For each $t\in[0,T]$ fixed, let $\mathcal{G}_t$ be the inverse of $\calF_t$. Let $\psi\in C_c\left( (0, T)\times\mathbb{R}^2\right) $ and set $\eta_t=\psi_t\circ\calG_t$. We note that 

\begin{equation}\label{eq:Conserv mass ASF 1}
\begin{aligned}
A &:=\int_0^T\int_{\mathbb{R}^2} \partial_t\eta+ \langle \nabla\eta,V\rangle d\sigma_t dt\\
  &=\int_0^T\int_{\mathbb{R}^2} \partial_t\psi_t\circ \calG_t+ \langle\nabla\psi_t\circ \calG_t, \partial_{t}\calG_t\rangle + \langle [\nabla\calG_t]^{\tau}\nabla\psi_t\circ\calG_t,V\rangle d\sigma_tdt\\
  &=\int_0^T\int_{\Gamma_{\varsigma_t}} \partial_t\psi_t+ \langle\nabla\psi_t, \partial_{t}\calG_t\circ \calF_t\rangle + \langle [\nabla\calG_t]^{\tau}\circ \calF_t\nabla\psi_t,V_t\circ\calF_t\rangle rdrdzdt.\\
 \end{aligned}
\end{equation}
 The last equality in \eqref{eq:Conserv mass ASF 1} is obtained by using $\calG_t\#\sigma_t=(r\chi_{\Gamma_{\varsigma_t}}\calL^2)$. The equations in \eqref{Def v,w} can be rewritten in the vectorial form as $ V\circ\calF_t=\partial_t\calF_t+ \nabla\calF_t\begin{bmatrix}v\\w \end{bmatrix} $ so that \eqref{eq:Conserv mass ASF 1} becomes
 \begin{equation}\label{eq:Conserv mass ASF 2}
\begin{aligned} 
  A&=\int_0^T\int_{\Gamma_{\varsigma_t}} r\partial_t\psi_t+ r\langle\nabla\psi_t, \partial_{t}\calG_t\circ \calF_t\rangle + r\langle [\nabla\calG_t]^{\tau}\circ \calF_t\nabla\psi_t,\partial_{t}\calF_t\rangle+ \langle [\nabla\calG_t]^{\tau}\circ \calF_t\nabla\psi_t,\nabla\calF_t \begin{bmatrix}rv\\rw \end{bmatrix}\rangle drdzdt\\
   &=\int_0^T\int_{\Gamma_{\varsigma_t}} r\partial_t\psi_t+ r\langle\nabla\psi_t, \partial_{t}\calG_t\circ \calF_t\rangle + r\langle \nabla\psi_t,[\nabla\calG_t]\circ \calF_t\partial_{t}\calF_t\rangle+ \langle \nabla\psi_t,[\nabla\calG_t]\circ \calF_t[\nabla\calF_t] \begin{bmatrix}rv\\rw \end{bmatrix}\rangle drdzdt\\
    &=\int_0^T\int_{\Gamma_{\varsigma_t}} r\partial_t\psi_t + \langle \nabla\psi_t, \begin{bmatrix}rv\\rw \end{bmatrix}\rangle drdzdt.\\
\end{aligned}
\end{equation}
In the second line of \eqref{eq:Conserv mass ASF 2}, we have used the fact that $\calG_t\circ\calF_t=\id$ implies that $  \partial_{t}\calG_t\circ \calF_t+[\nabla\calG_t]\circ \calF_t\partial_{t}\calF_t=0$ and $[\nabla\calG_t]\circ \calF_t[\nabla\calF_t]=\mathbf{I}$. Applying the divergence theorem in space-time, we obtain that

 \begin{equation}\label{eq:Conserv mass ASF 3}
\begin{aligned} 
  A&=-\int_0^T\int_{\Gamma_{\varsigma_t}} \psi_t\div(rv,rw) drdzdt+ \int_0^T\int_{\Lambda_{rig}} \psi_t\langle{{\bf u}}_t , {\bf n}_t \rangle d\mathbf{H}^2\\
    & +\int_0^T\int_{\varsigma_t=0}\frac{\psi_t}{\sqrt{|\partial_t\varsigma|^2+|\partial_z\varsigma|^2+1}}\left(\partial_t\varsigma_t + w\partial_z\varsigma_t-v \right)  d\mathbf{H}^2.\\
\end{aligned}
\end{equation}
Here $\mathbf{H}^2$ denotes the 2-dimensional Hausdorff measure. As $A=0$ and $\psi$ is arbitrary, \eqref{eq:Conserv mass ASF 3} implies that $v,w, \varsigma$ solve (\ref{eq1: conservation of mass physical}).\endproof\\
 We define  ${\bf s}:\calW_{\infty}\longrightarrow \mathcal{W}$ by ${\bf s}(r,z)=\left(s[r],z \right)  $ where $ 2s= 2s[r]:= r_0^{-2}-r^{-2}$. ${\bf s}$ is invertible with inverse
 $ {\bf r}$ defined by ${\bf r}(s,z)= \left(\frac{1}{\Omega}\sqrt{2f_0(s)}, z\right)$. If $\varrho\in\calR_0$ and $\varsigma$ is defined by the third equation of \eqref{eq : physical var 1} then 
 \begin{equation}\label{eq:varrho-varsigma}
  {\bf r}\#\mu_{\varrho}= r\chi_{\Gamma_{\varsigma}}\calL^2.
 \end{equation}
$\Gamma_{\varsigma}$  is defined in \eqref{eq: boundary cond 1}. For $ \theta_0$ fixed, we define $   \mathbf{s}_{\theta_0}$ on $\calW_{\infty}$ by  $\mathbf{s}_{\theta_0}(s,z)=(s[r],z,\theta_0(z)).$ To any $P\in C\left( \calW\times\calI_0\right) $ such that $P(\cdot, z,m)$ and $P(s, \cdot,m)$ are differentiable we  associate the function  $\mathcal{T}[P]$ defined by
 \begin{equation}\label{eq: subgradient1 P}
 \mathcal{T}[P](\p)=\left( \partial_{s}P(\p,\theta_0(z)), \;\theta_0(z)\partial_{z}P(\p,\theta_0(z))\right) \qquad \p=(s,z).
 \end{equation}
Similarly, to any  function $\Psi\in C(B_l^+)$ such that $\Psi(\Upsilon,\cdot)$ and $\Psi(\cdot,Z)$ are differentiable, we associate  the function $\mathcal{S}[\Psi]$ defined by
 \begin{equation}\label{eq: subgradient1 Psi}
 \mathcal{S}[\Psi](\q)=\left(\partial_{\Upsilon}\Psi(\q), \;\phi^{-1}(\partial_{Z}\Psi) (\q)\right)\qquad\quad\qquad \phi(z)= z/\theta_0(z),
 \end{equation}
\begin{proposition}
Let $l>0$ and assume that (A1) holds. Let $T>0$ and  $\sigma \in AC_2\left( 0, T; \mathscr{P}\left(  \mathbb{R}^2\right) \right)$  and $V\in C^1\left( (0,T)\times \mathbb{R}^2\right) $ such that 
$$\frac{\partial\sigma}{\partial t}+ \div (\sigma V_t)=0, \quad \mathcal{D}'\left((0,T)\times  \mathbb{R}^2\right).$$
 Let $P \in C^{1}\left((0,T)\times \overset{o}{\calW}\times \calI_0 \right) $,  $\Psi\in C^{1}\left( (0,T)\times B_l^+ \right) $ and $\varrho \in C^{1}\left((0,T)\times (0,H) \right) $  such that $\varrho_t\in \calR_0$ for each $t$ fixed. Assume that  $\mathcal{T} [P_t]$ and  $\mathcal{S} [\Psi_t]$ as defined in 
respectively in \eqref{eq: subgradient1 P} and \eqref{eq: subgradient1 Psi} are inverse of each other in the interior of their domains, that $\partial_sP>0$  and that $\mathcal{S} [\Psi_t^{\sigma}]$ pushes forward $\sigma_t$ onto $\mu_{\varrho_t}$ with
  \begin{equation}\label{eq : condition on free boundary on dual space} 
   2(1 -2r_0^2 \varrho_t(z))P_t(\varrho_t(z),z,\theta_0(z))=\Omega^2 r_0^2\qquad \left\lbrace\varrho_t>0\right\rbrace,
  \end{equation} 
  for each $ t\in[0, T]$. Define $\varphi$, $\varsigma$ and $ \theta$ respectively through \eqref{eq : physical var 1}  $u$  $v$ and $w$ through \eqref{eq : physical var 2} and set $\frac{D}{Dt}:= \partial_{t} + v\partial_{r} + w\partial_{z}.$
 Assume $V$ is the velocity field as in (\ref{eq1 unphysical1: velocity in dual space}). Then $(u, v, w)$, $\theta$, $\varphi$ and $\varsigma$ solve (\ref{eq:non physical})-(\ref{eq: bdary cond on pressure}) and (\ref{eq: stability condition}). 
\end{proposition}

\proof{}
The first equations of \eqref{eq : physical var 1} and \eqref{eq : physical var 2}  imply that
\begin{equation} \label{eq1: unphisical  define u theta varphi solution} 
 ( u r+r^2\Omega)^2=\partial_{s}P\circ {\bf s}_{\theta_0},\qquad\frac{g}{\theta_0}\theta=\partial_{z}P\circ {\bf s}_{\theta_0}.
\end{equation}
These, in light of the second equation of \eqref{eq : physical var 1}, yield (\ref{third1}) and (\ref{fourth1}). We define the function $\mathbb{T}=\left(\mathbb{T}_1, \mathbb{T}_2 \right)$ by $\mathbb{T}_t:= \calT[P_t]\circ {\bf s}$ and  $\mathbb{S}=\left(\mathbb{S}_1, \mathbb{S}_2 \right)$ by $\mathbb{S}_t:={\bf r}\circ \calS[\Psi_t]$ for each $t\in (0,T)$. As $\calT[P_t]$ and $ \calS[\Psi_t]$ are inverse of each other for each $t\in (0,T)$, so are $\mathbb{T}_t$ and $\mathbb{S}_t$ for $t\in (0,T)$. We notice that  we can rewrite  \eqref{eq1: unphisical  define u theta varphi solution} as
\begin{equation} \label{eq1: axisymmetric state} 
\mathbb{T}=\left(u r+\Omega r^2)^2, g\theta\right). 
\end{equation} 
Therefore, 
\begin{equation} \label{eq : Dt T1} 
 \begin{aligned}
\frac{D\mathbb{T}_1}{Dt}  &= 2( u r+\Omega r^2)\frac{D}{Dt}( u r+\Omega r^2)\\ 
                                 &=  2\sqrt{\mathbb{T}_1} (r\frac{{\rm D}  u}{{\rm D} t}+ u v +2r\Omega v )
\end{aligned}
\end{equation}
and 
\begin{equation} \label{eq : Dt T2} 
 \frac{D\mathbb{T}_2}{Dt}=   g\frac{D \theta}{Dt}. 
\end{equation}

  The two last equations of \eqref{eq : physical var 2} actually read in vectorial form
 \begin{equation}\label{eq velocity field 2}
 \begin{bmatrix}
v_t\\
w_t  
\end{bmatrix}
=\left( \partial_{t} \mathbb{S}_t\right)\circ \mathbb{T}_t + [ \nabla_{\Upsilon, Z}\mathbb{S}_t ]\circ \mathbb{T}_t
 \begin{bmatrix}
2\sqrt{\mathbb{T}_{1t}}F_{0t}\\ 
gF_{1t}
\end{bmatrix}
\end{equation}

In view of \eqref{eq velocity field 2},
\begin{equation}\label{eq velocity field 03}
[\nabla_{r,z}\mathbb{T}] \begin{bmatrix}
 v_t\\
 w_t
\end{bmatrix}= [\nabla_{r,z}\mathbb{T}_t]\left( \partial_{t} \mathbb{S}_t\right)\circ \mathbb{T}_t + [\nabla_{r,z}\mathbb{T}_t][\nabla_{\Upsilon,Z} \mathbb{S}_t]\circ \mathbb{T}_t
\begin{bmatrix}
2\sqrt{\mathbb{T}_1}F_{0t}\\ 
gF_{1t}
\end{bmatrix}.
\end{equation}
As
$$\left[ \dfrac{D\mathbb{T}_t}{Dt}\right]^{\tau} := \partial_t \mathbb{T} + v\partial_{r}\mathbb{T}+w\partial_{z}\mathbb{T} =  \partial_t \mathbb{T} + [\nabla_{r,z}\mathbb{T}]
 \begin{bmatrix}
 v\\
 w
\end{bmatrix},$$
we use \eqref{eq velocity field 03} to get
 \begin{equation}\label{eq velocity field 3}
\left[ \dfrac{D\mathbb{T}_t}{Dt}\right]^{\tau} = \partial_t \mathbb{T}_t+[\nabla_{r,z}\mathbb{T}_t]\left( \partial_{t} \mathbb{S}_t\right)\circ \mathbb{T}_t + [\nabla_{r,z}\mathbb{T}_t][\nabla_{\Upsilon,Z} \mathbb{S}_t]\circ\mathbb{T}_t
\begin{bmatrix}
2\sqrt{\mathbb{T}_1}F_0\\ 
gF_1
\end{bmatrix}.
\end{equation}

 Since $\mathbb{T}_t$ and $\mathbb{S}_t$ are inverse of each other for each $t\in(0,T)$, we have $\partial_t \mathbb{T}_t+[\nabla_{r,z}\mathbb{T}_t]\left( \partial_{t} \mathbb{S}_t\right)\circ \mathbb{T}_t=0$ and $[\nabla_{r,z}\mathbb{T}][\nabla_{\Upsilon,Z} \mathbb{S}_t]\circ\mathbb{T}_t={\bf I}$ so that  \eqref{eq velocity field 3} reduces to 
 \begin{equation}\label{eq velocity field 4}
\dfrac{D \mathbb{T}}{Dt}=\left(2\sqrt{\mathbb{T}_1}F_{0},\; gF_{1} \right).
\end{equation}
 
As $\partial_{s}P>0$ we have that $\mathbb{T}_1>0$. Subsequently, we combine \eqref{eq : Dt T1}, \eqref{eq : Dt T2}  and \eqref{eq velocity field 4} to obtain (\ref{first1}) and (\ref{second1}). In light of \eqref{eq:varrho-varsigma},  we observe that
\begin{equation}
\mathbb{S}\#\sigma_t={\bf r}\circ\mathcal{S}[\Psi_t]\#\sigma_t={\bf r}\#\mu_{\varrho_t}= r\chi_{\Gamma_{\varsigma_t}}\calL^2,
\end{equation}
for each $t \in [0, T]$. We use lemma \ref{le: conservation of total mass}  to obtain (\ref{fifth1}) and (\ref{eq: non physical free boundary}). We combine the second and third equations in \eqref{eq : physical var 1} with \eqref{eq : condition on free boundary on dual space} to get \eqref{eq: bdary cond on pressure}. The invertibility of $\mathbb{T}$ and the second equation in \eqref{eq : physical var 1} yield \eqref{eq: stability condition}. \endproof



\section{Minimization problem and Duality Method}
\label{sec:Duality Methods}

In this section, we prove the existence and uniqueness of a variational solution for problem \eqref{eq: main problem}.  This result is obtained by investigating c-subdifferential of maximizers in \eqref{eq:statement of dual problem} with respect to the cost function $c$ as defined in \eqref{eq: cost function} and by  establishing subsequently a duality between problem \eqref{eq:forK} and problem \eqref{eq:statement of dual problem}.\\\\
 Let $l>0$  and $\sigma\in \mathscr{P}\left(\mathbb{R}^2 \right) $ such that $\spt(\sigma)\subset B_l^+$. We consider the following system of equations where the unknowns are $P\in C(\calW\times \calI_0)$, $\Psi\in C(B_l)$ and $\varrho\in \calR_0$. We require that $P$ and $\Psi$ satisfy \eqref{lem:c-convex 1} and \eqref{lem:c-convex 2} and solve
\begin{equation}
 \begin{cases}\label{eq: main problem}
 \mathcal{T}[P]\#\mu_{\varrho}=\sigma, \\
 \mathcal{S}[\Psi]\circ \mathcal{T}[P]=\id \;\mu_{\sigma_{\varrho}}-a.e,\\
2(1-2r^2_{0} \varrho_0(z))P(\varrho_0(z),z, \theta_0(z))=r^2_0\Omega^{2} \text{ on } \{\varrho_0>0\}.
\end{cases}
\end{equation} 
\begin{remark}\label{def: equiv problem} The maps $\mathcal{S}[\Psi]$ and $\mathcal{T}[P]$ in \eqref{eq: main problem} are defined respectively in  \eqref{eq: subgradient1 P} and \eqref{eq: subgradient1 Psi}. We note that if $\sigma\in\mathscr{P}^{ac}\left(\mathbb{R}^2\right)$ then the system of equations \eqref{eq: main problem} is equivalent to
\begin{equation}
 \begin{cases}\label{eq: equiv problem}
 \mathcal{S}[\Psi]\#\sigma=\mu_{\varrho}, \\
\mathcal{T}[P]\circ \mathcal{S}[\Psi] =\id \;\sigma-a.e,\\
2(1-2r^2_{0} \varrho_0(z))P(\varrho_0(z),z, \theta_0(z))=r^2_0\Omega^{2} \text{ on } \{\varrho_0>0\}.
\end{cases}
\end{equation} 
\end{remark}

\subsection{Primal and  Dual formulation of the problem}\hfill

Let $\sigma\in \mathscr{P}\left(\mathbb{R}^2 \right) $. We define the functional $\calK[\sigma]$ on $\calR$ as follows :
\begin{equation}\label{eq:forK} 
\mathcal{K}[\sigma](\varrho) :=
\begin{cases}
\dfrac{1}{2} W^2_2\left(\sigma, {\bf f}\#\mu_{\varrho}\right) -\dfrac{1}{2}\int_{\mathbb{R}^2}\left(s^2+\frac{z^2}{\theta^2_0(z)} - {r_0^2 \Omega^2 \over 2(1-2 r_0^2 s)}\right)  \mu_{\varrho}(d\p)\qquad\qquad\qquad\qquad \varrho \in \calR_0\\

+\infty\qquad\qquad\qquad \qquad\qquad\qquad\qquad\qquad\qquad\qquad\qquad\qquad\qquad\qquad\quad\quad\;\varrho \notin \calR_0.
\end{cases} 
\end{equation}
Here, $\mathbf{f}$ is defined on $\mathcal{W}$ by $\mathbf{f}(s,z)=\left(s,\frac{z}{\theta_0(z)} \right)$. We consider the variational problem
\begin{equation}\label{eq:forK} 
\inf_{\varrho\in \calR} \left\lbrace \mathcal{K}(\varrho) :\quad \varrho \in \calR\right\rbrace.
\end{equation}
To study the minimization problem in (\ref{eq:forK}), we investigate a dual formulation through the functional
\begin{equation}\label{eq:forJ} 
\mathcal{J}[\sigma]( P,\Psi)= \int_{\mathbb{R}^2} - \Psi\;\sigma(dq) +\mathscr{H}(P); \quad 
\mathscr{H}(P)=\inf_{\varrho \in \calR} \int_0^H \mathscr{S}_{\theta_0}[P](\varrho(z),z)dz,
\end{equation}
 where $\calJ[\sigma]$ is defined on 
\begin{equation}\label{eq:for calU}
 \calU:=\Bigl\{ (P,\Psi)\in C( \bar{\mathcal{W}})\times C(\Reals^2_+) : P(\p,m)+\Psi(\q) \geq  c(\p,m,\q)\text{ for all } (\p,m,\q)\in\mathcal{W}\times\mathcal{I}_0\times\mathbb{R}_+^2 \Bigr\}
\end{equation}

and the functional $\mathscr{S}$ is defined by  
\begin{equation}\label{eq:Pi1} 
\mathscr{S}_{\theta_0}[P](\varrho,z)=\int_0^\varrho \Bigl( f_0(s)-P(s,z,\theta_0(z)) \Bigr) ( 2f_0(s)/\Omega^2)^2ds  \quad 
\end{equation}
 for $ (\varrho,z)\in \calW.$
 The dual problem we will be looking at is the following:
\begin{equation}\label{eq:statement of dual problem} 
 \sup\left\lbrace  \calJ[\sigma](P,\Psi): (P,\Psi)\in\calU\right\rbrace.
\end{equation}

Set 
\begin{equation}\label{eq: quadratic F}
F(s,t,\Upsilon,Z)=  1/2\vert s-\Upsilon\vert^2 + 1/2\vert t -Z\vert^2
\end{equation}
and
\begin{equation}\label{eq: change variables Phi}
 \Phi(\p,m,\q)= (s,z/m,\Upsilon,Z)
\end{equation}
for all $\p=(s,z)\in\calW$, $m\in \calI_0$ and $\q=(\Upsilon,Z)\in \mathbb{R}^2_+$.  We recall that $$ c(\p,m,\q)= s\Upsilon+ zZ/m.$$
Note that the cost function $c$ can be expressed as
\begin{equation}\label{eq: cost in terms of F}
 c(\p,m,\q) = F\circ\Phi(\p,m,\0)+ F\circ\Phi(\0,m,\q) -F\circ\Phi(\p,m,\q)  
\end{equation} 
 and the second moment of $\sigma$ is given by $$\mathbf{m}_2[\sigma]=\int_{\mathbb{R}^2} F\circ\Phi(\0,m,\q)\;\sigma(d\q).$$
\begin{proposition}\label{prop:optimality cond}
Let $l>0$, $\sigma\in \mathscr{P}\left( \mathbb{R}^2\right)$ and assume that (A1) holds. Then,
\begin{enumerate}
\item We have $\calK[\sigma](\varrho)\geq \calJ[\sigma](P,\Psi)+ \mathbf{m}_2[\sigma]$ for all $\varrho\in \calR$ and all $(P, \Psi)\in \calU$
\item Let $(P_0,\Psi_0)\in \calU_0$ and $\varrho_0\in \calR_0$. Then, the following hold:

$\calK[\sigma](\varrho_0)=\calJ[\sigma](P_0,\Psi_0)+ \mathbf{m}_2[\sigma]$ if and only if there exists $\alpha_0 \in \mathscr{P}\left( \mathbb{R}^2\times  \mathbb{R}\times \mathbb{R}^2\right)$ such that $\pi^{1,2,3}\#\alpha_0=\left( \id,\theta_0\circ\pi^2\right) \#\mu_{\varrho_0}$ and  $\pi^{3,4}\#\alpha_0= \sigma$    for which  $P_0(\p,m)+\Psi_0(\q)= c(\p,m,\q)$  $\alpha_0-$ a.e  and
\begin{equation}\label{eq: equality in scrH}
 \mathscr{H}(P_0)= \int_0^H \mathscr{S}_{\theta_0}[P_0](\varrho_0(z),z)dz.
 \end{equation}
 In that case, 
 $$W_2^2\left( \sigma, \mathbf{f}\#\mu_{\varrho_0} \right)=\int_{\mathbb{R}^2\times \mathbb{R}^2}|\p-\q|^2\; d\Phi\#\alpha_0. $$
\end{enumerate}

\end{proposition}
\proof{}
(1) Let $\varrho\in \calR_0$ and $(P, \Psi)\in C(\bar {\mathcal{W}}\times \bar\calI_0)\times C(\bar B_l^{+})$ be such that
\begin{equation}\label{eq:sup constraint}
 P(\p,m)+\Psi(\q) \geq  c(\p,m,\q)
 \end{equation} 

for all $(\p,m)\in   \mathcal{W}\times \calI_0$ and $\q\in  B_l^{+}$. Then,
\begin{equation}\label{eq: inequality in j}
\inf_{\bar\varrho\in\calR}\int_{\mathbb{R}^2\times \mathbb{R}} f_0(s)- P(\p,m) \; \delta_{\theta_0(z)}(dm)\mu_{\bar\varrho}(d\p) \leq\int_{\mathbb{R}^2\times \mathbb{R}} f_0(s)- P(\p,m) \; \delta_{\theta_0(z)}(dm)\mu_{\varrho}(d\p). 
\end{equation}
 This implies that
 \begin{equation}\label{eq: J leq I 1}
\begin{aligned} 
J[\sigma](P,\Psi)&= \int_{\mathbb{R}^2}-\Psi(\q)  \sigma(d\q) +\inf_{\bar\varrho\in\calR}\int_{\mathbb{R}^2\times \mathbb{R}} f_0(s)- P(\p,m) \; \delta_{\theta_0(z)}(dm)\mu_{\bar\varrho}(d\p) \\
                 &\leq  \int_{\mathbb{R}^2\times\mathbb{R}^2\times \mathbb{R}} f_0(s)-\Psi(\q)-P(\p,m)\; d\alpha \\          
                 &\leq\int_{\mathbb{R}^2\times  \mathbb{R}\times \mathbb{R}^2} f_0(s)- c(\p,m,\q)\;d\alpha         
\end{aligned}
 \end{equation} 
for any $\alpha \in \mathscr{P}\left( \mathbb{R}^2\times  \mathbb{R}\times \mathbb{R}^2\right)$ such that $\pi^{1,2,3}\#\alpha=\left( \id,\theta_0\circ\pi^2\right) \#\mu_{\varrho}$, $\pi^{4,5}\#\alpha= \sigma$. We have used \eqref{eq:sup constraint} in the last inequality of \eqref{eq: J leq I 1}. 
 \begin{equation}\label{eq: cost part1}
\int_{\mathbb{R}^2\times  \mathbb{R}\times \mathbb{R}^2} F\circ \Phi d\alpha= \int_{\mathbb{R}^2\times\mathbb{R}^2} F\; d\Phi\#\alpha
 \end{equation}
and  
\begin{equation}\label{eq: cost part2}
\int_{\mathbb{R}^2\times\mathbb{R}\times\mathbb{R}^2} F\circ\Phi(\p,m,\0)+F\circ\Phi(\0,m,\q) d\alpha=\int_{\mathbb{R}^2\times\mathbb{R}}F\circ\Phi(\p,m,\0)\delta_{\theta_0(z)}(dm)\mu_{\varrho}(d\p) + \int_{\mathbb{R}^2}F\circ\Phi(\0,m,\q) \sigma(d\q).
\end{equation}
 In view of \eqref{eq: cost in terms of F}, we combine \eqref{eq: J leq I 1}-\eqref{eq: cost part2} to obtain
 \begin{equation}\label{eq: J leq I 2}
 \calJ[\sigma](P,\Psi)\leq \int_{\mathbb{R}^2\times\mathbb{R}^2} F\; d\Phi\#\alpha+ \int_{\mathbb{R}^2} f_0(s)-F\circ\Phi(\p,m,\0)\delta_{\theta_0(z)}(dm) \mu_{\varrho}(d\p) - \int_{\mathbb{R}^2}F\circ\Phi(\0,m,\q) \sigma(d\q)
 \end{equation}
We observe that $\pi^{1,2}\#[\Phi\#\alpha]= \sigma$ and  $\pi^{3,4}\#[\Phi\#\alpha]={\bf f}\# \mu_{\varrho}$. By taking the infimum  in \eqref{eq: J leq I 2} over  $\alpha$ we obtain that
 \begin{equation}\label{eq: J leq I 2 W2}
 \calJ[\sigma](P,\Psi)\leq \frac{1}{2} W^2(\sigma,{\bf f}\#\mu_{\varrho})+ \int_{\mathbb{R}^2} f_0(s)- F\circ\Phi(\p,m,\0)\; \delta_{\theta_0(z)}(dm)\mu_{\varrho}(d\p) - \int_{\mathbb{R}^2}  F\circ\Phi(\0,m,\q)\; \sigma(d\q).
 \end{equation}
That is, 
 \begin{equation}\label{eq: J+m leq K}
 \mathcal{J}[\sigma](P,\Psi)+\mathbf{m}_2[\sigma ] \leq \calK[\sigma](\varrho).
 \end{equation}
(2) We have $(P,\Psi)\in \calU_0$ and $\varrho\in \calR$ satisfy the equality in \eqref{eq: J+m leq K} if and only if they satisfy the equality in \eqref{eq: inequality in j} and in the second line of \eqref{eq: J leq I 1}. The equality is satisfied in \eqref{eq: inequality in j} if and only if \eqref{eq: equality in scrH} holds. The equality  in the second line of \eqref{eq: J leq I 1} if and only if there exists  $\alpha_0 \in \mathscr{P}\left( \mathbb{R}^2\times  \mathbb{R}\times \mathbb{R}^2\right)$ such that $\pi^{1,2,3}\#\alpha_0=\delta_{\theta_0(z)}\mu_{\varrho}$ and $\pi^{4,5}\#\alpha_0= \sigma$ ,   and $ P_0(\p,m)+\Psi_0(\q) = c(\p,m,\q)$ hold $\alpha_0$ almost everywhere. In that case, the equality holds in \eqref{eq: J leq I 2} and then in \eqref{eq: J leq I 2 W2}. As a result,
  $$W_2^2\left( \sigma, \mathbf{f}\#\mu_{\varrho_0} \right)=  \int_{\mathbb{R}^2\times\mathbb{R}^2} 2F\; d\Phi\#\alpha_0=\int_{\mathbb{R}^2\times \mathbb{R}^2}|\p-\q|^2\; d\Phi\#\alpha_0. $$
\endproof
\subsection{c-transforms and c-subdifferentials}\label{subsection: Hypotheses}

\begin{definition}\label{def: c-transf}
 Let $l>0$, $\Psi\in C(B_l^{+})$ and $P\in C(\calW\times \calI_0)$. We define the $c-$transform of $\Psi$, denoted  $\Psi^c$,  by 
\begin{equation}\label{eq:conj P}
\quad\quad\;\;\Psi^c(\p,m)=\qquad \sup_{\q\in \bar B_l^+}\quad c(\p,m,\q)- \Psi(\q),\qquad (\p,m)\in\mathcal{W}\times\mathcal{I}_0.
 \end{equation}
Similarly, we define the $c-$transform of $P$, denoted $P_c$, by
\begin{equation}\label{eq:conj Psi}
 P_c(\q)= \sup_{(\p,m)\in\bar{\calW}\times\bar{\mathcal{I}_0 }}c(\p,m,\q)- P(\p,m), \qquad \q\in B_l.
\end{equation}

\end{definition} 
 We note that $c-$transform functions enjoy some regularity properties. Indeed, the functions $\Psi^c(p,\cdot), \Psi^c(\cdot,m)$ and $P_c$ are convex as supremum of convex functions. As a consequence, they are locally Lipschitz and thus differentiable almost everywhere with respect to the Lebesgue measure. We consider the set $\calU_0$ of functions $(P,\Psi)$ defined by 
\begin{equation}\label{lem:c-convex 1}
\qquad P(\p,m)= \qquad\sup_{\q\in \bar B_l^+}\qquad c(\p,m,\q)- \Psi(\q),\qquad (\p,m)\in\mathcal{W}\times\mathcal{I}_0 
 \end{equation}
and
\begin{equation}\label{lem:c-convex 2}
\Psi(\q)= \sup_{(\p,m)\in\mathcal{W}\times\mathcal{I}_0 } c(\p,m,\q)- P(\p,m), \qquad \q\in B_l.
\end{equation}
\begin{definition}
Let $l>0$ and $(P,\Psi)\in\calU_0$. For any $(\p,m)\in \mathcal{W}\times\mathcal{I}_0$, we define 
\begin{equation}
\partial^c P (\p,m)=\left\lbrace \q\in \bar B_l^+ : P(\p,m)+\Psi(\q)=c(\p,m,\q)\right\rbrace. 
\end{equation}

In a similar way, for any $\q\in B_l$ we define
\begin{equation}
\qquad\qquad\qquad\partial^c \Psi(\q)=\left\lbrace (\p,m)\in \bar{\mathcal{W}}\times\bar{\mathcal{I}}_0 : P(\p,m)+\Psi(\q)=c(\p,m,\q)\right\rbrace. 
\end{equation}
\end{definition} 

\begin{lemma}\label{le:sudiff structure}
Let $l>0$ and assume the condition  (A1) holds.
  \begin{enumerate}
  \item[(i)] There exists $k_0>0$ such that whenever $(P,\Psi)\in\calU_0$ we have $\Psi$ is $k_0-$Lipschitz continuous on $B_l$ and  $P$ is $k_0-$Lipschitz continuous on $\calW\times\calI_0$.
  \item[(ii)] Let $P$ be a $c$-transform of some $\Psi\in C(\bar B_l)$. If $m_0\in \mathcal{I}_0$ and $\p_0=(s_0,z_0)$ a point of differentiability of $P(\cdot,m_0)$ then 
  \begin{equation}\label{eq: subdiffP}
  \partial_c P (\p_0,m_0)= \left\lbrace\left(\partial_{s} P(\p_0,m_0),m_0\partial_{z} P(\p_0,m_0) \right)  \right\rbrace. 
  \end{equation}
    As a consequence, the function $ \mathcal{T}[P](\p)=\calA[\theta_0(z)]\nabla_{\p}P(\p,\theta_0(z))$ is well-defined Lebesgue almost everywhere. If, in addition, $P(\p_0,\cdot)$ is
 differentiable at $m_0$ then 
\begin{equation}\label{eq: diffm}
  m_0\partial_{m}P(\p_0,m_0)= -z_0\partial_{z}P(\p_0,m_0).
 \end{equation} 
 
  \item [(iii)] Let $\Psi$ be a $c$-transform of some $P\in C(\bar{\calW}\times\bar\calI_0)$.  If $\q_0$ be a point of differentiability of $\Psi$ then,
  \begin{equation}\label{eq: subdiffPsi}
  \partial^c \Psi (\q_0)= \left\lbrace (\p_0,m_0)\in \mathcal{W}\times\mathcal{I}_0:  \p_0=(s_0,z_0),\; s_0=\partial_{\Upsilon} \Psi(\q_0),\; z_0= m_0\partial_{Z}\Psi(\q_0)  \right\rbrace. 
  \end{equation}
  If we assume furthermore that $\q_0\in \partial_c P (\p_0, m_0)$, that $\p_0=(s_0,z_0) $ is a point of differentiability of $P (\cdot, m_0) $  and that $m_0=\theta_0(z_0)$ then the function  $ \mathcal{S}[\Psi]$ as defined in \eqref{eq: subgradient1 Psi} is defined  almost everywhere with respect to Lebesgue. 
  \end{enumerate}
\end{lemma} 
\proof{}
1. Let $\q_1, \q_2\in B_l$. Choose $(\p_1, m_1)\in \calW\times\mathcal{I}_0$  such that 
\begin{equation}\label{eq: Lip Psi1}
\Psi(\q_1)=  c(\p_1,m_1,\q_1)- P(\p_1,m_1).
\end{equation}
As $(\p_1, m_1)\in \calW\times\mathcal{I}_0$, we have
\begin{equation}\label{eq: Lip Psi2}
\Psi(\q_2)\geq  c(\p_1 ,m_1 ,\q_2)- P(\p_1, m_1).
\end{equation}
Set 
$$k_0 :=\max_{\substack{(\p,m)\in\bar{\mathcal{W}}\times \bar{\mathcal{I}_0}\\ \q\in \bar B_l}}\left\lbrace  |\nabla_{\q} c(\p,m,\q)| +|\nabla_{\p} c(\p,m, \q)|+|\partial_m c(\p,m, \q)| \right\rbrace. $$
We combine \eqref{eq: Lip Psi1} and \eqref{eq: Lip Psi1} to get
\begin{equation}
\Psi(\q_1)-\Psi(\q_2)\leq  c(\p_1, m_1,\q_1)- c(\p_1,m_1,\q_2)\leq k_0 |\q_2-\q_1|.
\end{equation}
 By permuting the roles of $\q_1$ and $\q_2$ is the above reasoning, we obtain that 
$$|\Psi(\q_2)-\Psi(\q_1)| \leq k_0 |\q_2-\q_1|.$$
 It follows that  $\Psi$ is $k_0$-lipschitz continuous on $B_l$. A similar argument shows that $P$ is $k_0$-lipschitz continuous on $\mathcal{W}\times \mathcal{I}_0$.\\
2. Let $m_0\in \mathcal{I}_0$ and $\p_0=(s_0,z_0)\in\calW$. Let $\q_0\in \partial_c P (\p_0,m_0)$, that is,
$$P(\p_0,m_0)= c(\p_0,m_0,\q_0)-\Psi(\q_0).$$
Consider the map $(u,t)\longrightarrow B(u,t):= c(u,t, \q_0)-P(u,t)-\Psi(\q_0)$. Assume that $\p_0=(s_0,z_0)\in\calW $ is a point of differentiability of $P(\cdot,m_0)$. Then, $B(\cdot, m_0)$ is differentiable at $\p_0$ and attained its maximum at $\p_0$. Thus, $\nabla_{\p}B(\p_0,m_0)=0$, that is, $\calA(1/m_0)\q_0=\nabla P(p_0,m_0)$. Hence $\q_0=\calA(m_0)\nabla P(\p_0,m_0)$. It follows that $\partial_c P (\p_0,m_0)$ is given by 
\eqref{eq: subdiffP}. Since $P(\cdot,m_0)$ is convex, it is locally lipschitz and thus differentiable Lebesgue  almost everywhere. This implies that the map $ \mathcal{T}[P](\p)=\calA[\theta_0(z)]\nabla_{\p}P(\p,\theta_0(z))$ is well-defined Lebesgue almost everywhere. Assume in addition that $P(p_0,\cdot)$ is differentiable at $m_0$ then $ \partial_{t}V(\p_0,m_0)=0$, that is,  \eqref{eq: diffm} holds.

3. Let $\q_0\in B_l$  and $(\p_0,m_0)\in \partial^c\Psi(\q_0)$ with $\p_0=(s_0,z_0)$. Then,  $P(p_0,m_0)= c(\p_0,m_0,\q_0)-\Psi(\q_0)$ and so the map $ \q\longrightarrow E(\q)=c(\p_0,m_0,\q)-P(\p_0,m_0)-\Psi(\q)$ attains its maximum at $\q_0$.
Assume $q_0$ is  a point of differentiability of $\Psi$. We have $\nabla E(\q_0)=0$, that is, $s_0=\partial_{\Upsilon}\Psi(\q_0)$ and $z_0=m_0\partial{Z}\Psi(\q_0)$. Thus, \eqref{eq: subdiffPsi} holds. Assume $m_0=\theta_0(z_0)$ and that (A1) holds. Since $\Psi$ is differentiable almost everywhere with respect to Lebesgue, the map $ \mathcal{S}[\Psi]$ is well-defined almost everywhere with respect to Lebesgue. $ \mathcal{S}[\Psi]$ is defined  almost everywhere with respect to Lebesgue.   
\endproof
\begin{remark} \label{rmk: partial P} Let $(P,\Psi)\in\calU_0$.
By the characterization of $\partial_c P$ provided  in lemma  \ref{le:sudiff structure}, we conclude that  $0\leq \partial_{z}P,\;\partial_s P\leq l\quad\calL^3-a.e$.
\end{remark}

\subsection{Existence of a maximizer in the dual problem}\hfill
\label{subsection: Hypotheses}

Let $l>0$. We recall that $\calU_0$ denotes the subset of $\calU$ consisting of $(P,\Psi)$ satisfying \eqref{lem:c-convex 1}-\eqref{lem:c-convex 2}.

For $A>0$ we denote by 
$$\mathcal{E}_A:=\left\lbrace P: (P,\Psi)\in \calU_0,\; \text{for some }  \Psi\;\text{and}\; P(\p,m)\leq A\; \text{for any  } (\p,m)\in\calW\times \mathcal{I}_0\right\rbrace. $$

\begin{lemma}\label{le:boundonheight1} Let $A>0$. Then
 \begin{equation}\label{eq:boundonheight1}
\lim_{\varrho\rightarrow 1/2r_0^2}\; \inf_{P\in \mathcal{E}_A}\inf_{0\leq z\leq H}\mathscr{S}_{\theta_0}[P](\varrho,z)=+\infty.
\end{equation}
\end{lemma} 
\begin{proof}{} As $P\leq A$ we have
\begin{equation}\label{eq:coercive0}
\mathscr{S}_{\theta_0}[P](\varrho, z)\geq \int_{0}^{\varrho} \Bigl( f_0(s)-A \Bigr) ( 2f_0(s)/\Omega^2)^2ds =\dfrac{r_0^4\varrho}{1-2r_0^2\varrho}\left[\dfrac{r_0^2\Omega^2(1- r_0^2\varrho)}{2(1-2r_0^2\varrho)}-Ar_0^4\right].
\end{equation} 

Since 

$$\lim_{\varrho\rightarrow 1/2r_0^2}\dfrac{r_0^4\varrho}{1-2r_0^2\varrho}=\lim_{\varrho\rightarrow 1/2r_0^2}\dfrac{r_0^2\Omega^2(1- r_0^2\varrho)}{2(1-2r_0^2\varrho)}=+\infty$$
 the result follows from \eqref{eq:coercive0}. 
\end{proof}

\begin{lemma}\label{lem: twist} Let $A>0$ and assume that conditions (A1) and (A2) hold.
\begin{enumerate}
\item[(i)] For each $z\in [0,H]$ fixed and $P\in \mathcal{E}_A$, $\mathscr{S}_{\theta_0}[P](\cdot,z)$ has a minimizer over $[0,1/(2r_0^2))$.
\item[(ii)] There exists $M_*$ such that the following holds : for any $P\in \mathcal{E}_A  $ and $z \in [0,H]$ if $\bar{\varrho}$ is a minimizer of $\mathscr{S}_{P}(\cdot,z)$ over $[0,1/(2r_0^2))$ then  
$$0\leq \bar\varrho\leq M_*\qquad  \hbox{and}\qquad 2r^2_{0}M_*<1.$$

\item[(iii)] Assume, additionally, that $P$ is Lipschitz continuous and that $\partial_{z}P \geq 0\quad \mathcal{L}^3- a.e $. Let $z_1, z_2 \in [0,H]$  and $\varrho_1, \varrho_2\in [0,1/(2r_0^2))$ be such that $\varrho_i$ is the minimizer of $\mathscr{S}_{\theta_0}[P](\cdot,z_i)$  over $ [0,1/(2r_0^2))$ $i=1,2$.    If $z_1\leq z_2$, then $\varrho_1 \leq \varrho_2.$ \\
\end{enumerate}
\end{lemma} 
\proof{} (i) follows from the continuity of  $\mathscr{S}_{\theta_0}[P](\cdot,z)$ over $[0, 1/(2r_0^2))$ and Lemma \ref{le:boundonheight1}.\\
(ii) Let $\left\lbrace M_n \right\rbrace_{n=1}^\infty $ be such that $0< M_n\leq 1/(2r_0^2)$ and $\left\lbrace M_n \right\rbrace_{n=1}^\infty $ converges to $1/(2r_0^2)$. Assume there exist  $\left\lbrace P_n \right\rbrace_{n=1}^\infty \subset\calE_A $ and  $\left\lbrace z_n \right\rbrace_{n=1}^\infty \subset[0,H] $ such that $\bar{\varrho}_n$ is a minimizer of $\mathscr{S}_{\theta_0}[P_n](\cdot,z_n)$ over $[0,1/(2r_0^2))$ but $M_n\leq\bar{\varrho}_n\leq 1/(2r_0^2)$. Then
 \begin{equation}
 \limsup_{n\rightarrow\infty} \mathscr{S}_{\theta_0}[P_n](\bar\varrho_n,z_n)\leq \mathscr{S}_{\theta_0}[P_n](0,z_n)=0.
 \end{equation}
This contradicts Lemma \ref{le:boundonheight1}.

(iii) For  $z \in [0,H]$ fixed, $\mathscr{S}_{\theta_0}[P](\cdot, z)$ is differentiable on $(0,1/(2r_0^2))$ and we have 
\begin{equation}\label{eq:boundonheight5} 
{\partial\mathscr{S}_{\theta_0}[P] \over \partial \varrho}(\varrho,z)= -( 2f_0(\varrho)/\Omega^2)^2 \left( f_0(\varrho)-P(\varrho,z,\theta_0(z))\right).
\end{equation}
Note that, by (A2), $(s,z)\longrightarrow (s,\theta_0(z))$ is Lipschitz continuous. As $P$ is Lipschitz continuous, $(s,z)\longrightarrow P(s,z,\theta(z))$ is Lipschitz and therefore differentiable  Lebesgue almost everywhere on $\calW$. The mixed partial derivatives of $\mathscr{S}_{\theta_0}[P]$  give
\begin{equation}\label{eq:twist1}
{\partial^2 \mathscr{S}_{\theta_0}[P] \over \partial z \partial \varrho}(\varrho,z)=-( 2f_0(\varrho)/\Omega^2)^2 \left(\partial_z P(\varrho,z,\theta_0(z)) + \partial_m P(\varrho,z,\theta_0(z))\theta^{'}_0(z)\right)
\end{equation} 
for a.e $(\varrho,z)\in\calW$. In light of \eqref{eq: diffm}, we have
\begin{equation}\label{eq:twist0}
\begin{aligned}
{\partial^2 \mathscr{S}_{\theta_0}[P] \over \partial z \partial \varrho}(\varrho,z)&= -( 2f_0(\varrho)/\Omega^2)^2 \left(\partial_z P(\varrho,z,\theta_0(z)) - z\partial_z P(\varrho,z,\theta_0(z))\frac{\theta^{'}_0}{\theta_0}(z)\right)\\
  &= -( 2f_0(\varrho)/\Omega^2)^2 \partial_z P(\varrho,z,\theta_0(z))\frac{\theta_0(z)-z\theta^{'}_0(z)}{\theta_0(z)}
\end{aligned}
\end{equation}
for a.e $(\varrho,z)\in\calW$. We recall that $\phi(z)=\frac{z}{\theta_0(z)}$ and  note that 
 \begin{equation}
\phi'(z)=\dfrac{\theta_0(z)-\theta_0'(z)z}{\theta_0^2(z)}
\end{equation}
for all $z\in(0,H)$. It follows that 
\begin{equation}\label{eq:twist1} 
{\partial^2 \mathscr{S}_{\theta_0}[P] \over \partial z \partial \varrho}(\varrho,z)= -( 2f_0(\varrho)/\Omega^2)^2 \partial_z P(\varrho,z,\theta_0(z))\theta_0(z) \phi'(z)
\end{equation} 
 for  almost every $(\varrho,z)\in\calW$. In light of the assumption (A1) and the fact that $\partial_{z}P \geq 0\quad \mathcal{L}^3- a.e$, we have  
 \begin{equation}\label{eq:twist2} 
{\partial^2 \mathscr{S}_{\theta_0}[P] \over \partial z \partial \varrho}(\varrho,z)\leq 0.
\end{equation} 
Let $z_i\in [0,H]$ and $\varrho_i$ be a minimizer $ \mathscr{S}_{\theta_0}[P](\cdot ,z_i)$ over $[0, 1/(2r_0^2))$, $i=1,2.$  We exploit this minimality  condition on $\varrho_1$, $\varrho_2$  to  obtain
\begin{equation}\label{eq:twist3} 
0 \leq \Bigl(\mathscr{S}_{\theta_0}[P](\varrho_2, z_1)-\mathscr{S}_{\theta_0}[P](\varrho_1,z_1) \Bigr) +\Bigl(\mathscr{S}_{\theta_0}[P](\varrho_1,z_2)-\mathscr{S}_{\theta_0}[P](\varrho_2,z_2)\Bigr) =-\int_{\varrho_1}^{\varrho_2} ds \int_{z_1}^{z_2} {\partial^2 \mathscr{S}_{\theta_0}[P] \over \partial z \partial s}(s,z) dz.
\end{equation} 
In light of \eqref{eq:twist2}, the equation \eqref{eq:twist3}  implies the following: if $z_1 <z_2,$ then   $\varrho_1 \leq \varrho_2.$ \endproof

\begin{lemma}\label{lem: varrho monotone}
Let $A>0$ and $P\in \calE_A$. Let $z_0\in(0,H)$. Assume that condition (A2) holds and that $P$ is Lipschitz continuous such that $\partial_{z}P\geq 0\quad\calL^3- a.e$. Let $\varrho_1, \varrho_2:  [0,H]\longrightarrow [0,1/(2r_0^2))$ be two maps defined in such a way that ${\varrho}_i(z)$ are minimizers of $\mathscr{S}_{\theta_0}[P](\cdot,z)$ over $[0,1/(2r_0^2))$.  Then, the following hold:
\begin{enumerate}
\item[(i)] $\varrho_1(z_0)$ and $\varrho_2$ are monotone. 
\item[(ii)] Assume ${\varrho}_i$ are continuous at $z_0$. Then,  $\varrho_1(z_0)=\varrho_2(z_0)$. 
\item[(iii)]  $\mathscr{S}_{\theta_0}[P](\cdot,z)$ has a unique minimizer over $[0,1/(2r_0^2))$  for almost every $z$ with respect to Lebesgue. 
\end{enumerate}
\end{lemma}  

\proof{} (i) follows from Lemma \eqref{lem: twist} (iii). Since ${\varrho}_1$ is continuous at $z_0$, $\lim_{\delta\rightarrow 0}\varrho_1(z_0-\delta)=\varrho_1(z_0)$. In light of Lemma \eqref{lem: twist} (iii), $\varrho_1(z_0-\delta)\leq \varrho_2(z_0)$ for $\delta$ small and positive. It follows that $\varrho_1(z_0)\leq \varrho_2(z_0)$. An analogous reasonning leads to $\varrho_1(z_0)\geq \varrho_2(z_0)$ which proves (ii). As $\varrho_1, \varrho_2$ are monotone, they have a countable number of discontinuous points. Thus, by (ii), $\varrho_1(z_0)=\varrho_2(z_0)$  for almost every $z$ with respect to Lebesgue. This proves (iii).
\endproof


\begin{lemma} \label{lem: compactness P,Psi} Let $l>0$ and $c_0\in\mathbb{R}$. Then, the following hold:
\begin{itemize}
\item[(i)] The set of all $(P,\Psi)\in \calU$  such that 
\begin{equation} \label{eq: super level}
(P,\Psi)\in \calU_0 \qquad \hbox{and}\qquad \mathcal{J}[\sigma](P,\Psi) \geq c_0
\end{equation}
 is precompact in $C(\bar\calW\times \bar \calI_0)\times C(\bar B_l)$.
\item[(ii)] The set $\calM$ of $P\in C(\bar\calW\times \bar \calI_0)$  such that $(P,\Psi)$ satisfies \eqref{eq: super level} for some $\Psi\in C(\bar B_l) $ is contained in  $\calE_A$ for some $A>0$.
\end{itemize}
\end{lemma} 
\proof{} Fix $(\frak{p},\frak{m})\in\calW\times \calI$ and  let $(P,\Psi)\in \calU_0$. By lemma \ref{le:sudiff structure}, $P$ is $k_0-$Lipschitz continuous. It follows that 
\begin{equation}\label{eq: bound P}
|P(\p,m)-P(\frak{p},\frak{m})|\leq k_0|\p-\frak{p}|+k_0|m- \frak{m}|=k_0\left( \frac{ H}{r_0^2}+2i_0\right) =:k_1 \quad  \text{with}\quad i_0=\calL^1(\calI_0),
\end{equation}
 for all $(\p,m)\in\calW\times \calI_0$. As $\Psi+P(\frak{p},\frak{m})=\left( P-P(\frak{p},\frak{m})\right)^c $ and  
 $\bar k_1:= \sup\left\lbrace c(\p,m,\q):(\p,m)\in\calW\times \calI_0,\; \q\in B_l \right\rbrace$ is finite, we have that 
 \begin{equation}\label{eq: bound Psi}
|\Psi(\q)+P(\frak{p},\frak{m})|\leq \bar k_1+ k_1=:k_2
\end{equation}
 for all $\q\in B_l$. We observe that 
 \begin{equation}\label{eq: estimate J}
 \mathcal{J}[\sigma](P,\Psi)\leq  -\int_{\mathbb{R}^2}\Psi d\sigma + \int_{0}^{H}\int_0^{\bar\varrho} \Bigl( f_0(s)-P(s,z,\theta_0(z)) \Bigr) ( 2f_0(s)/\Omega^2)^2ds dz
\end{equation}
 for $(P,\Psi)\in \calU_0$ and for any constant $\bar\varrho\in [0,1/(2r_0^2)))$. In light of \eqref{eq: bound P} and \eqref{eq: bound Psi}, the  estimate \eqref{eq: estimate J} implies 
$$
\mathcal{J}[\sigma](P,\Psi)\leq P(\frak{p}, \frak{m}) +k_2 + H\int_0^{\bar\varrho} \Bigl( f_0(s)+k_1-P(\frak{p}, \frak{m})  \Bigr) ( 2f_0(s)/\Omega^2)^2ds dz 
$$
for $(P,\Psi)\in \calU_0$ and for any constant $\bar\varrho\in [0,1/(2r_0^2)))$. For $(P,\Psi)\in \calU_0$  such that $\mathcal{J}[\sigma](P,\Psi) \geq c_0$, we have
$$
c_0\leq P(\frak{p}, \frak{m})\left( 1 -H\int_0^{\bar\varrho} ( 2f_0(s)/\Omega^2)^2ds dz \right)+ k_2 + H\int_0^{\bar\varrho} \Bigl( f_0(s)+k_1  \Bigr) ( 2f_0(s)/\Omega^2)^2ds dz 
$$
for any constant $\bar\varrho\in [0,1/(2r_0^2)))$. By choosing $\bar\varrho =0$  and then $\bar\varrho =\bar \varrho_0$ where $\bar\varrho$ is such that\\ $\left( 1 -H\int_0^{\bar\varrho_0} ( 2f_0(s)/\Omega^2)^2ds dz \right) <0$ we obtain 
\begin{equation}\label{eq: unif bound P}
c_0-k_2\leq P(\frak{p}, \frak{m}) \leq  A
\end{equation}

with 
\begin{equation}\label{eq: const A}
  A= \dfrac{c_0- k_2 - H\int_0^{\bar\varrho} \Bigl( f_0(s)-k_1 \Bigr) ( 2f_0(s)/\Omega^2)^2ds dz }{\left( 1 -H\int_0^{\bar\varrho} ( 2f_0(s)/\Omega^2)^2ds dz \right)}.
\end{equation}
As $(\frak{p},\frak{m})$ is an arbitrary point in $\calW\times \calI_0$, it follows from \eqref{eq: unif bound P} that the set $\calM$  is uniformly bounded  with respect to the uniform norm and, in particular \eqref{eq: unif bound P} implies that   $P\in\calE_A$ whenever $(P,\Psi)\in \calM$. This proves (ii). As $\calM$  is uniformly bounded  the estimate \eqref{eq: bound Psi} implies that the set of $\Psi$ such that $(P,\Psi)\in \calU_0$ and $\mathcal{J}[\sigma](P,\Psi) \geq c_0$  is uniformly bounded with respect to the uniform norm. Using the uniform Lipschitz constant established in lemma \ref{le:sudiff structure}  we have  that the set of  $(P,\Psi)\in \calU_0$ such that $\mathcal{J}[\sigma](P,\Psi) \geq c_0$ is precompact which proves (i). Whenever $(P,\Psi)\in \calU_0$, $P$ is Lipschitz -thus, differentiable Lebesgue almost everywhere- and $\partial_c P\subset B_l^+$. 
\endproof

\begin{lemma} \label{le:conv Jsigma(P Psi)}
 Let $l>0$ and $\left\lbrace \sigma_n\right\rbrace_{n=0}^{\infty}$ such that $\spt(\sigma_n)\subset B_l^+$ and $\left\lbrace \sigma_n\right\rbrace_{n=1}^{\infty}$ converges narrowly to $\sigma_0$. Let  $( P_0,\Psi_0)\subset\calU$ and $\left\lbrace( P_n,\Psi_n)\right\rbrace_{n=1}^{\infty}\subset\calU_0$ such $\left\lbrace P_n\right\rbrace_{n=1}^{\infty}$ converges uniformly to $P_0$ on compact subsets of $\calW\times \calI_0$ and $\left\lbrace \Psi_{n}\right\rbrace_{n=1}^{\infty}$  converges uniformly to $\Psi_0$ on $B_l$. Then, $\left\lbrace \calJ[\sigma_n]\left( P_n, \Psi_{n}\right) \right\rbrace_{n=1}^{\infty}$ converges to  $\calJ[\sigma_0]\left( P_0, \Psi_{0}\right)$.
\end{lemma}

\proof{} Let $\varrho_n\in \calR$  such that $\varrho_n$ is monotone and $\varrho_n(z)$ is the minimizer of  $\mathscr{S}_{\theta_0}[P_n](\cdot,z)$ over $[0, 1/(2r_0^2))$ for $n\geq 0$ and for each $z\in[0,H]$ fixed, as provided by lemma \ref{lem: twist}(i) and lemma \ref{lem: varrho monotone} (i). By lemma \ref{lem: twist}(ii) there exists $M_* >0$ such that 
$$0\leq 2r_0^2\varrho_n(z)\leq 2M_*r_0^2<1$$
 for $n\geq 0$ and $z\in[0,H]$. Helly's theorem ensures that $\left\lbrace \varrho_n \right\rbrace_{n=1}^{\infty} $ converges pointwise  -up to a subsequence denoted again $\left\lbrace \varrho_n \right\rbrace_{n=1}^{\infty} $- to $\bar\varrho$. We set $\calW_{M_*}= [0, M_*]\times[0,H]$. Note that $\mathscr{S}_{\theta_0}[P_n](\cdot,z)$ is uniformly bounded on $\calW_{M_*}$.  As $\left\lbrace P_n\right\rbrace_{n=1}^{\infty}$ converges uniformly to $P_0$ on  $\calW_{M_*}\times\calI_0$, we easily check that $\left\lbrace \mathscr{S}_{\theta_0}[P_n]\right\rbrace_{n=1}^{\infty}$ converges  uniformly to $\mathscr{S}_{\theta_0}[P_0]$ on  $\calW_{M_*}$. As a result, $\bar\varrho$  minimizes   $\mathscr{S}_{\theta_0}[P_0](\cdot,z)$ over $[0, 1/(2r_0^2))$. In light of lemma \ref{lem: varrho monotone} (iii), it follows that $\bar \varrho= \varrho_0$ almost everywhere with respect to Lebesgue. By the definition of $\varrho_n$, it is straightforward that $\calH(P_n)= \int_0^H\mathscr{S}_{\theta_0}[P_n](\varrho_n(z),z)dz$ and so, the Lebesgue dominated convergence ensures that  $\left\lbrace \calH(P_n)\right\rbrace_{n=1}^{\infty}$ converges to $\calH(P_0)$. Thus, 
\begin{equation}\label{eq: limsup Jn-J0}
\begin{aligned}
\limsup_{n\rightarrow\infty}\Big \vert \calJ[\sigma_n]\left( P_n, \Psi_{n}\right)-\calJ[\sigma_0]\left( P_0, \Psi_{0}\right)\Big \vert\leq \limsup_{n\rightarrow\infty}\Big \vert\int_{B_l^+} \Psi_n \sigma_n(d\q)-\int_{B_l^+} \Psi_0 \sigma_0(d\q)\Big \vert.
\end{aligned}
\end{equation}
 As $\left\lbrace \sigma_n\right\rbrace_{n=0}^{\infty}$ converges narrowly to $\sigma_0$ and  $\left\lbrace \Psi_n\right\rbrace_{n=1}^{\infty}$ converges uniformly to $\Psi_0$ on  $B_l^+$, we get that 
\begin{equation}\label{eq:limsup Psin}
\begin{aligned}
\limsup_{n\rightarrow\infty}\Big \vert\int_{B_l^+} \Psi_n \sigma_n(d\q)-&\int_{B_l^+} \Psi_0 \sigma_0(d\q)\Big \vert\\
&\leq \limsup_{n\rightarrow\infty}\int_{B_l^+} \Big \vert\Psi_n - \Psi_0\Big \vert\sigma_n(d\q)+\limsup_{n\rightarrow\infty}\Big \vert\int_{B_l^+} \Psi_0 \sigma_n(d\q)-\int_{B_l^+} \Psi_0 \sigma_0(d\q)\Big \vert
=0.
\end{aligned}
\end{equation}
It follows from \eqref{eq: limsup Jn-J0} and \eqref{eq:limsup Psin} that $\left\lbrace \calJ[\sigma_n]\left( P_n, \Psi_{n}\right) \right\rbrace_{n=1}^{\infty}$ converges to  $\calJ[\sigma_0]\left( P_0, \Psi_{0}\right)$.
\endproof

\begin{proposition}\label{Proposition MA} 
Let $l>0$ and $\sigma\in \mathscr{P}\left(\mathbb{R}^2 \right) $ such that $\spt(\sigma)\subset B_l$.  $\calJ[\sigma]$   admit a maximizer over $\calU_0$.

\end{proposition}
\proof{}
Note that $\calJ[\sigma]\not\equiv\infty$. Indeed, set 
$$\mathfrak{c}_0=\sup_{\mathcal{W}\times \mathcal{I}_0\times B_l}c(\p,m,\q)\qquad P_{00}=\mathfrak{c}_0/2 \qquad \hbox{and }\Psi_{00} =\mathfrak{c}_0/2.$$
Then, $(P_{00}, \Psi_{00})\in \calU$ and $\frak{c}_{00} :=\calJ[\sigma](P_{00},\Psi_{00})$ is finite. Let  $\{( P_n, \Psi_n)\}_{n=1}^{\infty} \subset \calU$ be a maximizing sequence of $\calJ[\sigma]$. One can easily check that $ P_n\leq (P_{n\;c})^c$, $ \Psi_n\leq (P_{n\;c})$ and that $\calJ[\sigma](P_n,\Psi_n)\leq \calJ[\sigma](P_{n\;c})^c,P_{n\;c})$. As $ \{(P_{n\;c})^c, (P_{n\;c})\}\in\calU_0$, we assume without loss of generality that the maximizing sequence $\{( P_n, \Psi_n)\}_{n=1}^{\infty} \subset \calU_0$. Therefore, $\mathcal{J}[\sigma]( P_n, \Psi_n) >\frak{c}_{00} $ for $ n\geq n_0$ for some positive integer $n_0$. In light of lemma \ref{lem: compactness P,Psi}, there exists a subsequence of $\{( P_n, \Psi_n)\}_{n=1}^\infty $ that we denote again by $\{( P_n, \Psi_n)\}_{n=1}^\infty$ that   converges uniformly to $  (P_0,\Psi_0)$. By lemma \ref{le:conv Jsigma(P Psi)}, we have that $\{\calJ[\sigma](P_n, \Psi_n)\}_{n=1}^\infty $  converges to $ \calJ[\sigma](P_0,\Psi_0)$. As a result, $(P_0,\Psi_0) $ is a maximizer of $\calJ[\sigma]$ over $\calU$ and we have $\calJ[\sigma](P_0,\Psi_0)\leq \calJ[\sigma](P_{0\;c})^c,P_{0\;c})$. This concludes the proof \endproof
 
\subsection{Existence of a minimizer in the primal problem}\hfill\\
In this section, we show the existence and uniqueness of the minimizer in variational problem. This result is achieved through the study of dual problem. Subsequently, we obtain a solution for problem \eqref{eq: main problem}.
\begin{proposition}\label{change of variables} 
Let $c_0,\;l>0$ and $\sigma\in \mathscr{P}(\mathbb{R}^2)$ such that $\spt(\sigma)\subset B^+_l$. Assume the condition (A1) and (A2) hold.
\begin{enumerate}
\item[(i)] $\calK[\sigma]$ admits a unique minimizer $\varrho_0$ over $\calR_0$. Furthermore,  if $(P_0,\Psi_0)\in \calU_0$ is a maximizer of $\calJ[\sigma]$ on $\calU$, then $\calT[P_0]:= \calA(\theta_0)\nabla P_0$ pushes $\mu_{\varrho_0}$ forward onto $\sigma$ so that $\calJ[\sigma](P_0,\Psi_0)=\calK[\sigma](\varrho_0)+ \mathbf{m}_2[\sigma] $ and $\varrho_0$ is monotone non decreasing on $[0,H]$  satisfying 
\begin{equation} \label{eq 1: prop: P on the boundary}
2(1-2r^2_{0} \varrho_0(z))P_0(\varrho_0(z),z, \theta_0(z))=r^2_0\Omega^{2} \text{ on } \{\varrho_0>0\}.
\end{equation}
If, additionally, we assume that $\sigma$ is absolutely continuous with respect to the Lebesgue measure then $\calS[\Psi_0]$, defined in \eqref{eq: subgradient1 Psi}, pushes $\sigma$  forward onto  $\mu_{\varrho_0}$ and we have
\begin{equation} \label{eq 1: inversibility of the gradient of P}
\mathcal{S}[\Psi_0]\circ\mathcal{T}[P_0]= \id \quad  \mu_{\varrho_0} \quad a.e  \quad \mathcal{T}[P_0]\circ\mathcal{S}[\Psi_0]= \id \quad a.e \quad \sigma.
\end{equation}
 \item[(ii)]  Assume  $\sigma$ is absolutely continuous with respect to the Lebesgue measure such that $\frac{\partial\sigma}{\partial\calL^2}>c_0\quad\calL^2-a.e$ and that $\spt(\sigma)=B_l^+$. If $(P_0,\Psi_0),\;(P_1,\Psi_1)\in\calU_0$ are  such that $(P_0,\Psi_0)$ is a maximizer of $\calJ[\sigma]$ and $\calJ[\sigma](P_0,\Psi_0)= \calJ[\sigma](P_1,\Psi_1)$ then we have that $P_1=P_0 \quad\text{on}\quad \calW\times \calI_0 $ and $\Psi_1=\Psi_0\quad \text{on}\quad B_l^+$.
 
 \item[(iii)] Assume that (A1') holds and that $(P_0,\Psi_0)$ is a maximizer of $\calJ[\sigma]$ such that $\partial_{z} P_0\geq b_0\quad \calL^3-a.e$ for some $b_0>0$. For any $z_1,z_2\in[0,H]$ such that $\varrho_0(z_1), \varrho_0(z_2) >0$,  there exists $C>0$ such that
 \begin{equation}\label{eq: Lip boundary D}
 |z_2-z_1|\leq C|\varrho_0(z_2)-\varrho_0(z_1)|.
 \end{equation}
\end{enumerate}
\end{proposition}

\begin{remark}\hfill
\begin{itemize}
\item  If $\spt(\sigma)\subset B_l^{b_0+}:=B_l^+\cap\calV_{b_0} $ where $ \calV_{b_0}=(0,\infty)\times(b_0,\infty)$ with $0<b_0<l$  then $\calJ[\sigma]$ admits a maximizer $(P_0,\Psi_0)$ satisfying \eqref{lem:c-convex 1} and \eqref{lem:c-convex 2} with $B_l^+$ replaced by  $B_l^{b_0+}$.  As a result, we obtain $\partial_{z} P_0\geq b_0\quad \calL^3-a.e$.
\item The estimate \eqref{eq: Lip boundary D}  implies that  the boundary of the domain $D_{\varrho_0}$ is piecewise Lipschitz  continuous. This result can be found in \cite{CulSed2014}.
\end{itemize}

\end{remark}
\proof{}
1.
Assume that   $(P_0, \Psi_0)\in\calU_{0}$ is a maximizer of $\calJ[\sigma]$ over $\calU$. Let  $\varrho_0\in \calR$  such that for each $z\in[0,H]$  $\varrho_0(z)$ is a minimizer of $\mathscr{S}_{\theta_0}[P_{0}](\cdot,z)$ over $[0,1/(2r_0^2))$. Then, if $\varrho_0(z)>0$ by differentiating $\mathscr{S}_{\theta_0}[P_{0}](\cdot,z)$ at $\varrho_0(z)$ we get  \eqref{eq 1: prop: P on the boundary}. Using the minimizing property of $\varrho_0$, we have
 $$\int_0^H\mathscr{S}_{\theta_0}[P_{0}](\varrho_0(z),z)dz\leq \int_0^H \mathscr{S}_{\theta_0}[P_{0}](\varrho(z),z)dz$$
 for all $\varrho\in\calR$. As a result,
 \begin{equation}\label{eq:scrH0}
 \mathscr{H}(P_0)= \int_0^H \mathscr{S}_{\theta_0}[P_0](\varrho_0(z),z)dz.
 \end{equation}
 For $ h \in C_c(\Reals^2)$ and $\kappa \in (-1,1)$, we set
 $$\Psi_\kappa=\Psi_0 + \kappa h \qquad \hbox{ and }\qquad  P_{\kappa}(\p,\theta_0(z))=\inf_{\q\in B_l^+}\left\lbrace c(p,\theta_0(z),\q)-\Psi_0(\q) - \kappa h(\q)\right\rbrace.$$ 
 We note that $ \left\lbrace P_{\kappa}\right\rbrace_{-1<\kappa<1}\subset C(\bar\calW\times \calI_0). $ One can show  that (cfr\cite{Gangbo95}) the following holds:
\begin{equation}\label{eq:controlPe1} 
||P_\kappa-P_0||_\infty \leq |\kappa| ||h||_\infty \qquad \hbox{ and }\qquad  \lim_{\kappa \rightarrow 0} {P_\kappa(\p,\theta_0(z))-P_0(\p,\theta_0(z)) \over \kappa}= -h(\calT[P_0](\p)) 
\end{equation} 
for Lebesgue almost every $\p \in \mathbb{R}^2$. Let $\left\lbrace \kappa_{n}\right\rbrace^{\infty} _{n=1}$ a sequence of $(-1,1)$  that converges to $0$. Let $\varrho_{\kappa_n}\in\calR $ such that $\varrho_{\kappa_n}(z)$ a minimizer of $\mathscr{S}_{\theta_0}[P_{\kappa_n}](\cdot,z)$ over $[0,1/(2r_0^2))$. It follows from lemma \ref{lem: twist} (i) and lemma \ref{lem: varrho monotone} (i) that $\left\lbrace \varrho_{\kappa_n} \right\rbrace_{n=1}^{\infty} $ is a sequence of monotone functions of $\calR$ uniformly  bounded away from $1/(2r_0^2)$. By Helly's theorem there exists a subsequence of $\left\lbrace \varrho_{\kappa_n} \right\rbrace_{n=1}^{\infty} $ still denoted $\left\lbrace \varrho_{\kappa_n} \right\rbrace_{n=1}^{\infty} $ such that $\left\lbrace \varrho_{\kappa_n} \right\rbrace_{n=1}^{\infty} $ converges to some $\bar\varrho\in \calR$. In view of the first equation of \eqref{eq:controlPe1}, $\left\lbrace P_{\kappa_n} \right\rbrace_{n=1}^{\infty} $ is a sequence of continuous functions that  converges uniformly to $P_0$ on compact subsets of $\calW\times \calI_0$  and so, $\left\lbrace\mathscr{S}_{\theta_0}[P_{\kappa_n}] \right\rbrace_{n=1}^{\infty} $ is a sequence of continuous functions that converges uniformly to $\mathscr{S}_{\theta_0}[P_{0}]$ on compact subsets of $\calW$. As a result, $\bar\varrho$ is a  minimizer of $\mathscr{S}_{\theta_0}[P_{\kappa_0}](\cdot,z)$ over $[0,1/(2r_0^2))$. Using lemma \ref{lem: varrho monotone} (iii), we conclude that $\varrho_0=\bar\varrho$ Lebesque almost everywhere on $[0,H]$. And so,
\begin{equation}\label{eq:limitrho} 
\lim_{\kappa \rightarrow 0} \varrho_\kappa(z)= \varrho_0(z)
\end{equation} 

 for Lebesgue almost all $z\in[0,H]$. We exploit the minimizing property of $\varrho_0(z)$ to get
\begin{equation}\label{eq:minimizerrho1} 
\begin{aligned}
\mathscr{S}_{\theta_0}[P_0](\varrho_0(z),z)-\mathscr{S}_{\theta_0}[P_\kappa](\varrho_\kappa(z),z)&\leq  \mathscr{S}_{\theta_0}[P_0](\varrho_\kappa(z),z)-\mathscr{S}_{\theta_0}[P_\kappa](\rho_\kappa(z),z)\\
&=    \int_0^{\varrho_\kappa(z)} (P_\kappa(s,z)-P_0(s,z))\frac{4f^2_0(s)}{\Omega^2}ds.
\end{aligned}
\end{equation} 
 Analogously, we use the minimizing property of $\varrho_{\kappa}(z)$ to obtain 
\begin{equation}\label{eq:minimizerrho2}
\begin{aligned}
\mathscr{S}_{\theta_0}[P_\kappa](\varrho_\kappa(z),z)-\mathscr{S}_{\theta_0}[P_0](\varrho_0(z),z)&\leq \mathscr{S}_{\theta_0}[P_\kappa](\varrho_0(z),z)-\mathscr{S}_{\theta_0}[P_0](\varrho_0(z),z)\\
 &=   - \int_0^{\varrho_0(z)} (P_\kappa(s,z)-P_0(s,z))\frac{4f^2_0(s)}{\Omega^2} ds.
\end{aligned}
\end{equation} 
We combine \eqref{eq:minimizerrho1} and \eqref{eq:minimizerrho2} to get that 
\begin{equation}\label{eq:minimizerrho3}
a(\kappa)\leq \mathscr{H}(P_0)-\mathscr{H}(P_{\kappa})\leq b(\kappa)
\end{equation}
with 
$$ a(\kappa)=\int_0^H \int_0^{\varrho_0(z)} (P_\kappa(s,z)-P_0(s,z))\frac{4f^2_0(s)}{\Omega^2} ds dz \quad\hbox{and} \quad  b(\kappa)=\int_0^H \int_0^{\varrho_{\kappa}(z)} (P_\kappa(s,z)-P_0(s,z))\frac{4f^2_0(s)}{\Omega^2} ds dz  $$
 By lemma \ref{lem: twist}, we  choose $M_*$ such that  
\begin{equation}\label{eq:hdeltabounded}
0\leq 2r_0\varrho_0(z),2r_0\varrho_{\kappa_n}(z)\leq  2r_0^2 M_*<1
\end{equation}
for  $z\in[0,H]$ and $n\geq 1$. As $f_0$ is bounded on $[0,M_*]$,
\begin{equation}\label{eq:scrH2}
\begin{aligned}
 \Big|b_{\kappa}-a_{\kappa}\Big|&\leq \Big|\int_0^H dz\int_{\varrho_0(z)}^{\varrho_{\kappa}(z)} \left( P_{\kappa}(s,z,\theta_0(z))-P_0(s,z,\theta_0(z))\right) \frac{4f^2_0(s)}{\Omega^2} ds \Big|\\
                                &\leq |\kappa|\max_{[0,M_*]}\frac{4f^2_0(s)}{\Omega^2}||h||_{\infty}\int_0^H |\varrho_{\kappa}(z)-\varrho_0(z)|dz.
 \end{aligned}
\end{equation}
It follows that
\begin{equation}\label{eq:scrH3}
 \limsup_{\kappa\rightarrow 0} \frac{1}{|\kappa|}\Big|b_{\kappa}-a_{\kappa}\Big|\leq \max_{[0,M_*]}\frac{4f^2_0(s)}{\Omega^2}||h||_{\infty}\limsup_{\kappa\rightarrow 0}\int_0^H |\varrho_{\kappa}(z)-\varrho_0(z)|dz=0.
\end{equation}

By the Lebesgue dominated convergence theorem, (\ref{eq:controlPe1}) implies that 
\begin{equation}\label{eq:scrH4}
\begin{aligned}
\lim_{\kappa\rightarrow 0}\left( a_{\kappa}/\kappa\right)& =\lim_{\kappa\rightarrow 0} \int_0^H dz\int_0^{\varrho_0(z)} \dfrac{(P_{\kappa}(s,z,\theta_0(z))-P_0(s,z,\theta_0(z))}{\kappa} \frac{4f^2_0(s)}{\Omega^2}ds \\
&=-\int_0^H \int_0^{\varrho_0(z)}h(\calT[P](p)) \frac{4f^2_0(s)}{\Omega^2}dsdz.
\end{aligned}
\end{equation}
We combine \eqref{eq:scrH3} and \eqref{eq:scrH4} to get that 
\begin{equation}\label{eqlimH}
\lim_{\kappa\rightarrow 0}\dfrac{\mathscr{H}(P_{\kappa})-\mathscr{H}(P_0)}{\kappa}= -\int_{\mathbb{R}^2} h(\calT[P_0](p))d\mu_{\varrho_0}.
\end{equation}

We note that 
\begin{equation}\label{eq: variational ratio}
{\mathcal{J}[\sigma](P_{\kappa}, \Psi_{\kappa})-\mathcal{J}[\sigma](P_0, \Psi_0) \over \kappa}=  -\int_{B_l} h d\sigma + \dfrac{\mathscr{H}(P_{\kappa})- \mathscr{H}(P_0)}{\kappa}.
\end{equation}
 We use \eqref{eqlimH} and \eqref{eq: variational ratio} to get 
\begin{equation}\label{eq:limitJje11}   
\lim_{\kappa \rightarrow 0}  {\mathcal{J} [\sigma](P_\kappa, \Psi_\kappa)-\mathcal{J}[\sigma](P_0, \Psi_0) \over \kappa}= -\int_{B_l} h d\sigma +  \int_{\mathbb{R}^2}  h(\calT[P_0])\; d\mu_{\varrho_0}. 
\end{equation}  
Since $(P_0,\Psi_0)$ maximizes $\mathcal{J}[\sigma]$ over $\calU$  and  $(P_\kappa, \Psi_\kappa) \in \calU $, (\ref{eq:limitJje11}) implies that  
\begin{equation}\label{eq:limitJje12}   
\int_{B_l} h d\sigma=\int_{\mathbb{R}^2}  h(\calT[P_0])\; d\mu_{\varrho_0}. 
\end{equation} 
 As $h \in C_c(\Reals^2)$ is arbitrary, we have that (\ref{eq:limitJje12}) implies that $\calT[P_0]\# \mu_{\varrho_{0}}=\sigma$. By lemma \ref{le:sudiff structure}, $\calT[P_0](\p)\in \partial^c P_0(\p,\theta_0(z))$ for almost every $\p\in\calW$. As $\mu_{\varrho_0}$ is absolutely continuous with respect to Lebesgue, we have 
 \begin{equation} \label{eq: Opt with P}
P_0(\p,\theta_0(z))+\Psi_0(\calT[P_0](\p))= c(\p,\theta_0(z),\calT[P_0](\p))\;\qquad \mu_{\varrho_0}-a.e \end{equation}
  that is, 
  $$P_0(\p,m)+\Psi_0(\q)= c(\p,m,\q)\; \alpha_0-a.e\; \quad \text{where} \quad \alpha_0=\left(\id,\theta_0\circ\pi^2,\calT[P_0] \right)\#\mu_{\varrho_0}. $$ 
  This, combined with \eqref{eq:scrH0} yields $\calK[\sigma](\varrho_0)+ \mathbf{m}_2[\sigma]=\calJ[\sigma](P_0,\Psi_0)$ in light of Proposition \ref{prop:optimality cond}. As a result, $\varrho_0$ is a minimizer of $\calK[\sigma]$ over $\calR_0$.

5. Assume $\sigma$ is absolutely continuous with respect to Lebesgue. A similar reasoning as above yields that $\calS[\Psi_0]\# \sigma=\mu_{\varrho_{0}}$. As $\left( \calS[\Psi_0](\q), \theta_0\circ\phi^{-1}(\partial_Z\Psi_0(\q))\right) $ belongs to $\partial^c\Psi(\q)$ for Lebesgue almost every $\q\in B_l$ we have
\begin{equation} \label{eq: Opt with Psi}
P_0\left( \calS[\Psi_0](\q), \theta_0\circ\phi^{-1}(\partial_Z\Psi_0(\q))\right)+\Psi_0(\q)= c\left( \calS[\Psi_0](\q), \theta_0\circ\phi^{-1}(\partial_Z\Psi_0(\q)),\q\right)\;\qquad \sigma-a.e. 
\end{equation}
 Using lemma \ref{le:sudiff structure}, the results in  \eqref{eq: Opt with P} and \eqref{eq: Opt with Psi} imply that $ \calT[P_0]\circ\calS[\Psi_0](\q)=\q \quad\sigma-a.e$ and $\calS[\Psi_0]\circ \calT[P_0](\p)= \p\quad \mu_{\varrho_0}-a.e$. It follows that $\alpha_0=\left(\calS[\Psi_0],\;\theta_0\circ\calS_1[\Psi_0],\; \id  \right)\#\sigma$. We note that 
 \begin{equation}\label{eq : unique minimizer measure 1}
\Phi\#\alpha_0=\left(\mathbf{f}\circ\calS[\Psi_0],\;\id \right)\#\sigma. 
 \end{equation}
By lemma \ref{prop:optimality cond}, $\Phi\#\alpha_0$ is the unique optimal plan between $\sigma$ and $\mathbf{f}\#\mu_{\varrho_0} $ with respect to the quadratic distance.
6. Assume $(P_1,\Psi_1)$ is another maximizer of $\calJ[\sigma]$ in $\calU_0$. In light of \eqref{eq : unique minimizer measure 1}, we have
 \begin{equation}\label{eq : unique minimizer measure 2}
 \Phi\#\alpha_0=\left(\mathbf{f}\circ\calS[\Psi_0],\;\id \right)\#\sigma=\left(\mathbf{f}\circ\calS[\Psi_1],\;\id \right)\#\sigma. 
  \end{equation}

$\Phi$ here is defined in \eqref{eq: change variables Phi}. As $\mathbf{f}$ is bijective, \eqref{eq : unique minimizer measure 2} implies that $\calS[\Psi_0]=\calS[\Psi_1]\quad \sigma-$a.e,  that is, $\partial_{\Upsilon}\Psi_0=\partial_{\Upsilon}\Psi_1$ and $\partial_{Z}\Psi_0=\partial_{Z}\Psi_1\;\sigma-$a.e. As $\sigma$ is absolutely continuous with respect to the Lebesgue measure with $\frac{\partial\sigma}{\partial\calL^2}>c_0\quad\calL^2-a.e$ and $\Psi_0,\;\Psi_1$ are Lipschitz continuous, we have $\Psi_1=\Psi_0+k$ on $B_l$ for some $k\in\mathbb{R}$. Since $(P_0,\Psi_0), (P_1,\Psi_1)\in\calU_0$ we get  
 \begin{equation}\label{eq : unique potentials}
P_1=\Psi_1^c=  \left( \Psi_0+k\right)^c =   \Psi_0^c -k=   P_0-k\qquad \hbox{on}\qquad \calW\times\calI_0.
\end{equation}
In light of \eqref{eq 1: prop: P on the boundary}, the equation \eqref{eq : unique potentials} yields that $P_1=P_0$.

7. Set $Q(s,z)=f_0(s)-P_0(s,z,\theta_0(z))$. Then, by lemma \ref{lem: compactness P,Psi} (iii)
 \begin{equation}\label{eq : estimate partial Q 1}
\partial_s Q(s,z)=f'_0(s)-\partial_s P_0(s,z,\theta_0(z))\leq ||f'||_{L^{\infty}[0,M_*]}=:c_1
\end{equation}

and
 \begin{equation}\label{eq : estimate partial Q 2}
\partial_z Q(s,z)=- \partial_z P_0(s,z,\theta_0(z))-\theta_0'(z)\partial_m P_0(s,z,\theta_0(z))=\theta_0(z)\phi'(z) \partial_z P_0(s,z,\theta_0(z)).
\end{equation}

 We use \eqref{eq : estimate partial Q 1} to obtain
 \begin{equation}\label{eq: Lips Q1}
Q(\varrho_0(z_2),z_1)-Q(\varrho_0(z_1),z_1)=\int_{\varrho_0(z_1)}^{\varrho_0(z_2)}\partial_s Q(\bar s, z_1)d\bar s \leq c_2(\varrho_0(z_2)-\varrho_0(z_1)).
\end{equation} 
We recall that $\theta_0$ has values in the bounded interval $\calI_0$. By condition (A1'), there exists $b_1>0$ such that $\theta_0\geq b_1$ and $\phi'\geq b_1$. Thus,  \eqref{eq : estimate partial Q 2} implies that $\partial_zQ(\bar s, z_1)\geq b_0 b_1^2=:c_2$. It follows that
\begin{equation}\label{eq: Lips Q2}
Q(\varrho_0(z_2),z_1)-Q(\varrho_0(z_2),z_2)=\int_{z_1}^{z_2}\partial_z Q(\varrho_0(z_2), z)dz \geq c_2(z_2-z_1).
\end{equation} 
 In view of \eqref{eq 1: prop: P on the boundary},  $Q(\varrho_0(z_1),z_1)= Q(\varrho_0(z_2),z_2)=0$ so that by combining \eqref{eq: Lips Q1} and \eqref{eq: Lips Q2}, we obtain 
 \begin{equation}\label{eq: Lips Q12}
(z_2-z_1)\leq  \frac{c_1}{c_2}(\varrho_0(z_2)-\varrho_0(z_1))
\end{equation} 
We obtain \eqref{eq: Lip boundary D} by interchanging $z_1$ and $z_2$ in \eqref{eq: Lips Q12}.
\endproof


\section{ Stability of the optimal transports }
\label{sec: Continuity property of the optimal transport}
Let $l>0$, $\sigma\in \mathscr{P}\left(\mathbb{R}^2\right) $ such that $\spt(\sigma)\subset B_l$ and $\theta_0:[0,H]\longrightarrow \calI_0$. We recall that 
\begin{equation}\label{eq:forK0} 
\mathcal{K}[\sigma](\varrho) =\dfrac{1}{2} W^2_2\left(\sigma, {\bf f}\#\mu_{\varrho}\right) +\dfrac{1}{2}\int_{\mathbb{R}^2}\left( f_0(s)-s^2-\phi^2(z)\right)  \mu_{\varrho}(d\p), \qquad \phi(z)=z/\theta_0(z)
\end{equation}
for any   $\varrho\in\calR_0$, the set of all $\varrho$ for  which $\mu_{\varrho}$ is a probability measure.  Here, $\mathbf{f}(s,z)=\left(s, z/\theta_0(z) \right) $  for any $(s,z)\in\calW$.
 As, $\theta_0$ is of values in $\calI_0$, $\phi$ is bounded and
\begin{equation}\label{eq:forK bounded} 
|\mathcal{K}[\sigma]|(\varrho)\leq  KH \max_{\bar{\calW}}|\mathbf{f}|+ 2\left(1/(4r_0^4)+\max_{[0,H]}\phi \right) +l^2=:C_0(K),
\end{equation}
for all $\varrho\in \calR_0$ such that $0\leq 2r_0^2\varrho(z)\leq  2r_0^2 K<1$  for all $z\in[0,H]$.

\begin{lemma}\label{lem: stability K[sigma]}
Let $l>0$ and  $\left\lbrace \sigma_n\right\rbrace_{n=0}^{\infty}\subset\mathscr{P}\left(\mathbb{R}^2\right)$ such that $\spt(\sigma_n)\subset B_l$ for all $n\geq 0$ and $\left\lbrace \varrho_n\right\rbrace_{n=0}^{\infty}\subset\mathcal{R}_0$. Assume that $\left\lbrace \sigma_n\right\rbrace_{n=1}^{\infty}$ converges narrowly to $\sigma_0$ and that $\left\lbrace \varrho_n\right\rbrace_{n=1}^{\infty}$ converges pointwise to $\varrho_0$. Then, $\left\lbrace\mathcal{K}[\sigma_n](\varrho_n)\right\rbrace_{n=1}^{\infty}$ converges to $\mathcal{K}[\sigma_0](\varrho_0)$.
\end{lemma}
\proof{} As $\mathbf{f}$ is bounded continuous and $\left\lbrace \mu_{\varrho_n}\right\rbrace_{n=1}^{\infty}$ converges narrowly to $\mu_{\varrho_0}$, we have that $\mathbf{f}\#\mu_{\varrho_n}$ is supported in a fixed bounded domain for $n\geq 1$ and $\left\lbrace \mathbf{f}\#\mu_{\varrho_n}\right\rbrace_{n=0}^{\infty}$ converges narrowly to $\mathbf{f}\#\mu_{\varrho_0}$. We then use the continuity of the Wasserstein distance $W_2(\cdot,\cdot)$ to get the result. \endproof

\begin{proposition}\label{prop: stability transport}
Let $c_0\in\mathbb{R}$, $l>0$ and  $\left\lbrace \sigma_n\right\rbrace_{n=0}^{\infty}\subset\mathscr{P}\left(\mathbb{R}^2\right)$ such that $\spt(\sigma_n)\subset B_l$ for all $n\geq 0$.  Let $\left\lbrace(P_n,\Psi_n)\right\rbrace_{n=0}^{\infty}\in\calU_0$ such that $\calJ[\sigma_n](P_n,\Psi_n)\geq c_0$ and let $\left\lbrace \varrho_n\right\rbrace_{n=0}^{\infty}\subset \mathcal{R}_0$ be a sequence of monotone functions such that
\begin{equation}\label{eq: stability equation 0}
\calK[\sigma_n](\varrho_n)+ {\bf m}_2[\sigma_n]=\calJ[\sigma_n](P_n,\Psi_n)
\end{equation}
for all $n \geq 0$. If $\left\lbrace \sigma_n\right\rbrace_{n=1}^{\infty}$ converges narrowly to $\sigma_0$ then the following holds :
\begin{itemize}
\item[(i)] $\left\lbrace \mu_{\varrho_n}\right\rbrace_{n=1}^{\infty}$ converges narrowly to $\mu_{\varrho_0}$. 
\item[(ii)] $\left\lbrace \calT[P_n]\right\rbrace_{n=1}^{\infty}$ converges pointwise to $\calT[P_0]$ Lebesgue almost everywhere.
\item [(iii)]  $\left\lbrace \calS[\Psi_n]\right\rbrace_{n=1}^{\infty}$ converges pointwise to $\calS[\Psi_0]$ Lebesgue almost everywhere.
\end{itemize}
 
\end{proposition}
\proof{} 1. Lemma \ref{lem: compactness P,Psi} ensures that $P_n\in \calE_A$ for all $n\geq 0$ for some $A>0$. Using lemma \ref{lem: twist}(ii), there exists $M_*>0$ such that 
\begin{equation}
0\leq 2r_0^2\varrho_n(z)\leq 2r_0^2M_*<1.
\end{equation}
for all $n\geq 0$ and $z\in[0, H]$.  In light of Helly's theorem, we assume that $\left\lbrace \varrho_n\right\rbrace_{n=1}^{\infty}$ converges to some monotone function $\bar\varrho_0$. As a result, it is straightforward that $\left\lbrace \mu_{\varrho_n}\right\rbrace_{n=1}^{\infty}$ converges weakly$^*$ to $\mu_{\bar\varrho_0}$. Note that
$$  \int_{\mathbb{R}^2}f_0(s)\mu_{\varrho_n}(d\p)+\dfrac{1}{2}W_2^2(\sigma_n, {\bf f}\#\mu_{\varrho_n})=\calK[\sigma_n](\varrho_n)+\int_{\mathbb{R}^2}s^2+\phi^2(z)\;\mu_{\varrho_n}(d\p).$$
And so, 
$$ \int_{\mathbb{R}^2}f_0(s)\mu_{\varrho_n}(d\p)\leq \calK[\sigma_n](\varrho_n)+\int_{\mathbb{R}^2}s^2+\phi^2(z)\mu_{\varrho_n}(d\p)\leq C_0(M_*) +1/(4r_0^4)+\max_{[0,H]}\phi. $$
Thus,   $\left\lbrace \mu_{\varrho_n}\right\rbrace_{n=1}^{\infty}$ is tight and without loss of generality, we assume that  $\left\lbrace \mu_{\varrho_n}\right\rbrace_{n=1}^{\infty}$ converges narrowly to $\mu_{\bar\varrho_0}$. We next show that $\varrho_0=\bar\varrho_0\quad\calL^1-a.e$. In light of lemma \ref{le:conv Jsigma(P Psi)} and lemma \ref{lem: stability K[sigma]}, \eqref{eq: stability equation 0} becomes in the limit:
\begin{equation}\label{eq: stability equation 1}
\calK[\sigma_0](\bar\varrho_0)+ {\bf m}_2[\sigma_0]=\calJ[\sigma_0](P_0,\Psi_0).
\end{equation}
In view of lemma \ref{prop:optimality cond}, the equality in \eqref{eq: stability equation 1} implies that $\bar\varrho_0$ is a minimizer of $\calK[\sigma_0]$. The uniqueness result established in proposision \ref{change of variables} thus guarantees that $\bar\varrho_0=\varrho_0\quad\calL^1-a.e$.  The reasoning above applies to any subsequence of $\left\lbrace \mu_{\varrho_n}\right\rbrace^{\infty}_{n=1}$. As the limit is unique, we conclude that (i) holds. \\
2. Let $\p_0=(s_0,z_0)$ be a point of $\calW$ such that $P_n$ is differentiable  at $(\p_0,\theta(z_0))$  for $n\geq 0$. Let $\q_n\in\partial P_n(\p_0,\theta_0(z_0))$. As $(\p_0,\theta(z_0))$ is a point of differentiability of $P_n$, we have  $\q_n=\calT[P_n](\p_0)$ by lemma \ref{le:sudiff structure}(ii). Since $\partial P_n(\p_0,\theta_0(z_0))\subset \bar B_l$,  up to a subsequence, $\left\lbrace \q_n\right\rbrace^{\infty}_{n=1}$ converges to some $\q_0\in \bar{B}_l$. By definition of $\partial P_n(\p_0,\theta_0(z_0))$, we have $P_n(\p_0,\theta_0(z_0))+\Psi_n(\q_n)=c(\p_0,m,\q_n)$. The continuity of $\Psi_n$ and $c$ and uniform convergence of  $\left\lbrace P_n\right\rbrace^{\infty}_{n=1}$  and  $\left\lbrace \Psi_n\right\rbrace^{\infty}_{n=1}$ yield  $P_0(\p_0,\theta_0(z_0))+\Psi_0(\q_0)=c(\p_0,m,\q_0)$. Thus, $\q_0=\calT[P_0](\p_0)$, which is independent of  subsequences of $\left\lbrace \q_n\right\rbrace^{\infty}_{n=1}$. As $P_n, \; n\geq 0$  is differentiable almost everywhere, (ii) holds. (iii) holds by similar arguments. \endproof


\section{Existence of solutions for  Continuity equations associated with the Axisymmetric  Model}
\label{sec:Continuity equation for the Toy Model}
In section \ref{sec: Derivation of the continuity equation}, we identified a class of continuity equations which yield solutions to the axisymmetric flows provided the velocity field associated with this of continuity equations is smooth enough. In this section, we construct solutions to such continuity equations. We point out, however, that the solution constructed are not smooth enough to generate a solution to the axisymmetric flow.

Assume (A1) holds and let $T>0$. For any $\Psi:\calV\longrightarrow \mathbb{R}$ convex such that $\nabla\Psi(\q)\in \calW$ a.e,  for all $\q\in\calV$ and $t\in[0,T]$ we associate the velocity field 
\begin{equation}\label{eq: velocity field in dual space} 
  V_t[\Psi]= {\textstyle \left(\;2\sqrt{\Upsilon} F_{0t}\left( \frac{1}{\Omega}\sqrt{2f_0(\frac{\partial \Psi}{\partial \Upsilon})},\phi^{-1}\left( \frac{\partial\Psi}{\partial Z}\right)  \right)  ,\; g F_{1t}\left( \frac{1}{\Omega}\sqrt{2f_0(\frac{\partial \Psi}{\partial \Upsilon})},\phi^{-1}\left( \frac{\partial\Psi}{\partial Z}\right) \right)  \right) }\quad\hbox{where}\quad \phi(z)=z/\theta_0(z).
 \end{equation}
 Let $l>0$. Under condition (B1), 
 \begin{equation}\label{eq: velocity bound} 
 |V_t[\Psi](\q)|\leq M\sqrt{4l+1}=:C_0(l)
\end{equation}
 for all $t\in[0,T]$ and $\q\in B_l$.
\begin{lemma}\label{lem: convergence in distribution when measures abs continuous} 
Assume conditions (A1), (B1), (B2) and (B3) hold. Let $l>0$ and $\Psi:\calV\longrightarrow \mathbb{R}$ convex such that $\nabla\Psi(\q)\in \calW$ a.e for all $\q\in\calV$. There exists a sequence of convex smooth functions  $\left\lbrace \Psi_{n}\right\rbrace_{n\geq 1}$ $\Psi_n: B_l\longrightarrow \mathbb{R}$ such that $\div\left( V_t[\Psi_{n}]\right)\geq 0 $ and  $\nabla\Psi_n(\q)\in \calW$ a.e for all $\q\in B_l$ for all $n \geq 1$ and $\left\lbrace  V_t[\Psi_{n}]\right\rbrace_{n\geq 1}$ converges to $  V_t[\Psi]$ almost everywhere with respect to the Lebesgue measure.
\end{lemma}
\proof{}
Since $\Psi$ be a convex function, $\Psi$ is locally Lipschitz and thus differentiable almost everywhere with respect to Lebesgue. Let $j$ be a smooth probability density contained with support contained in the unit ball.  We consider the functions  $\Psi_n: \calV_n\longrightarrow \mathbb{R}$ defined by $\Psi_n= j_n*\Psi$, with $ j_n= \frac{1}{n^2}j(\frac{\cdot}{n}) $ and $\calV_n=\left\lbrace \q\in \calV : dist(\q, \partial \calV) \right\rbrace $.  It follows that $\left\lbrace \nabla\Psi_n\right\rbrace_{n=1}^{\infty} $ converges to $\nabla\Psi$  in $L_{loc}^1(\calV)$. Thus, there exists a subsequence of $\left\lbrace \nabla\Psi_n\right\rbrace_{n=1}^{\infty} $ denoted again by 
$\left\lbrace \nabla\Psi_n\right\rbrace_{n=1}^{\infty} $ that  converges to $\nabla\Psi $ almost everywhere with respect to the Lebesgue measure. As $\phi^{-1}$ is continuous, $\left\lbrace\phi^{-1}\left( \frac{\partial\Psi_n}{\partial Z} \right)\right\rbrace_{n=1}^{\infty} $ converges to $\left\lbrace\phi^{-1}\left( \frac{\partial\Psi}{\partial Z} \right)\right\rbrace$ almost everywhere with respect to the Lebesgue measure. 
As a consequence, $\left\lbrace  V_t[\Psi_{n}]\right\rbrace_{n\geq 1}$ converges to $  V_t[\Psi]$ almost everywhere with respect to the Lebesgue measure. As $\Psi_n$ is smooth, we have
\begin{equation}\label{eq: div}
\begin{aligned}
\div[V_t[\Psi_n]]=\frac{1}{\sqrt{\Upsilon}}F_0 &+\frac{1}{\sqrt{2}\Omega}\partial_{r}F_0\frac{f'_0}{2\sqrt{f_0}}\partial^2_{\Upsilon\Upsilon}\Psi_n\\
& +2\sqrt{\Upsilon}\partial_{z}F_0\frac{1}{\phi'\circ\phi^{-1}}\partial_{Z\Upsilon}^2\Psi_n+\frac{g}{\sqrt{2}\Omega}\partial_{r}F_1\frac{f'_0}{2\sqrt{f_0}}\partial^2_{\Upsilon Z}\Psi_n +\partial_{z}F_1\frac{g}{\phi'\circ\phi^{-1}}\partial_{ZZ}^2\Psi_n.
\end{aligned}
\end{equation}
Here, for simplicity, we make the following identifications: 
$$F_i\equiv F_{it}\left( \frac{1}{\Omega}\sqrt{2f_0(\frac{\partial \Psi}{\partial \Upsilon})},\phi^{-1}\left( \frac{\partial\Psi}{\partial Z}\right)  \right),\qquad f_0\equiv  f_0(\frac{\partial \Psi}{\partial \Upsilon}),\qquad 
\phi^{-1} \equiv \phi^{-1}\left( \frac{\partial\Psi}{\partial Z}\right)  $$
$$\partial_{r}F_i\equiv \partial_{r} F_{it}\left( \frac{1}{\Omega}\sqrt{2f_0(\frac{\partial \Psi}{\partial \Upsilon})},\phi^{-1}\left( \frac{\partial\Psi}{\partial Z}\right)  \right),\qquad f_0^{'}\equiv  f_0^{'}(\frac{\partial \Psi}{\partial \Upsilon}),\qquad \partial_{z}F_i\equiv \partial_{z} F_{it}\left( \frac{1}{\Omega}\sqrt{2f_0(\frac{\partial \Psi}{\partial \Upsilon})},\phi^{-1}\left( \frac{\partial\Psi}{\partial Z}\right)  \right) $$
 $i=0,1$. In light of  the convexiy of $\Psi_n$ and conditions (A1), (B1), (B2) and (B3), the equation \eqref{eq: div} implies that $\div\ [V_t[\Psi_n]] \geq 0$. \endproof
\begin{lemma}\label{lem: Cont.Eq time step} 
Let $l_0>0$, $\tau>0$. $C_0$ is as defined in \eqref{eq: velocity bound}. Assume conditions (A1), (A2), (B1), (B2) and (B3). Let $t_0 >0$ and $\sigma_{t_0}\in \mathscr{P}^{ac}\left( \mathbb{R}^2\right) $ such that $\spt(\sigma_{t_0})\subset B^+_{l_0}$. Let $\Psi$ a convex function on $\calV$ such that $\nabla\Psi(\q)\in \calW$ $\calL^2$-a.e,  for all $\q\in\calV$. Then, there exists $\sigma_t\in \mathscr{P}^{ac}\left( \mathbb{R}^2\right) $ such that $\spt(\sigma_{t})\subset B^+_{l_t}$ with $l_t\leq l_0+ C_0(l_0)(t-t_0)$ for  $t\in[t_0,t_0+\tau) $ satisfying :\\
\begin{itemize}
\item[(a)] $\int_{\mathbb{R}^2}\left(  \frac{\partial \sigma_t}{\partial \calL^2}\right)^r d\q\leq \int_{\mathbb{R}^2} \left(  \frac{\partial \sigma_{t_0}}{\partial \calL^2}\right)^{r}d\q $  for any  $ r\geq 1 $ and $t\in [t_0, t_0+\tau)$.\\
\item[(b)] $t\longmapsto\sigma_t\in AC_{1}\left( t_0, t_0+\tau ;\mathscr{P}( \mathbb{R}^2)\right) $ and 
\begin{equation}\label{eq:cont.Eq0} 
\begin{cases}
 \frac{\partial\sigma}{\partial t}+ \div (\sigma V_t[\Psi])=0, \qquad  \mathcal{D}'\left( (t_0,\;  t_0+\tau)\times\mathbb{R}^2 \right) \\
\sigma_{|t=t_0}=\bar\sigma_{t_0}.
\end{cases}
\end{equation}

\item[(c)] $t\longmapsto\sigma_t$ is Lipschitz continuous with respect to the $1-$Wasserstein distance and satisfies
\begin{equation}\label{eq :cont. Eq.} 
W_1\left( \sigma_t, \sigma_{\bar t}\right) \leq C_0(l_0) |t-\bar{t}|
\end{equation}
for all $t_0\leq \bar{t}, t\leq t_0+\tau_0$.
\end{itemize}
\end{lemma}

\proof{}
By lemma \ref{lem: convergence in distribution when measures abs continuous},  there exist a sequence $\left\lbrace \Psi_n    
\right\rbrace_{n=1}^{\infty}$ such that $\div(V_t[\Psi_n]) \geq 0$ and $ V_t[\Psi_n]$ converges to $V_t[\Psi]$ $\calL^2$-a.e. For each $n$ fixed, let $w^n_t$ be the flow associated to the vector field $V_t[\Psi_n]$ defined by $\dot{w}^n_t=V_t[\Psi_n](w^n)$ and $w_{t_0}^n(t_0)=\id$. Then, $\sigma^n_t=w^n_t\#\bar{\sigma}_{t_0}$ solves $\eqref{eq:cont.Eq0}$ when $\Psi$  is replaced by  $\Psi_n$.  Since $\div(V_t[\Psi_n]) \geq 0$, we have $\det\left( \nabla w^n\right)\geq \det\left( \nabla w^n_{t_0}\right)=1$ for any  $n\geq 1$, and $t\in [t_0, t_0+\tau)$. It  follows that

\begin{equation}\label{eq:L^r} 
\int_{\mathbb{R}^2}\left(  \frac{\partial \sigma^n}{\partial \calL^2}\right)^r d\q=\int_{\mathbb{R}^2}\left(\dfrac{ \frac{\partial \sigma_{t_0}}{\partial \calL^2}}{\det\left( \nabla w^n\right)}\right)^{r}\circ( w^{n})^{-1}d\q \leq \int_{\mathbb{R}^2} \left(  \frac{\partial \sigma_{t_0}}{\partial \calL^2}\right)^{r}d\q
\end{equation}
for any  $n\geq 1$, $ r\geq 1 $ and $t\in [t_0,\; t_0+\tau)$. This ensures that (a) holds for $\sigma^n$. In view of \eqref{eq: velocity bound}, we have 
$$\dfrac{d}{dt}|w^n|= \dot{w}_t^n\cdot\frac{w_t^n}{|w|}= V_t[\Psi_n](w^n)\cdot\frac{w_t^n}{|w^n_t|}\leq |V_t[\Psi_n]|\leq C_0(l_0).$$
Therefore, 
$$|w_t^n(\q)| \leq |\q|+ C_0(l_0)(t-t_0)$$
for all $t\in [t_0, t_0+\tau)$ and $\q\in B_{l_0}^+$. It follows that $w_t^n(B_{l_0}^+)\subset B_{l_{t}}^+ $ where $l_{t}\leq l_0+ C_0(l_0)(t-t_0)$. As  $\spt(\sigma_{t_0}) \subset B_{l_0}^+$ and $w^n_t$ is continuous, we have that $\spt(\sigma_{t}) \subset B_{l_t}^+$ 
 for all $t\in[t_0, t_0+\tau) $. By  [Theorem 8.3.1,\;\cite{ags:book}],
\begin{equation}\label{eq:lip W1} 
 W_1(\sigma^n_t, \sigma^n_{\bar{t}})\leq \int_{\bar t}^t ||V_t[\Psi_n]_{r}||_{L^1(\sigma_r )}dr \leq C_0(l_0) (t-\bar{t}) \qquad  \text{for all } t_0\leq \bar{t} \leq t \leq t_0+\tau.
\end{equation}
Consequently,  $t\longrightarrow\sigma^n_t$ is $C_0(l_0)$-Lipschitz continuous on $[t_0,\;t_0+\tau)$ for all $n\geq 1$. Thus,
\begin{equation}\label{eq: unif bounded  W_1} 
   W_1(\bar\sigma^n_{t_0}, \sigma^n_t)\leq C_0(l_0)(t-t_0)\leq C_0(l_0)\tau
  \end{equation}
for all $t\in [t_0,t_0+\tau]$ $n\geq 1$ . We conclude that $\left\lbrace \sigma^n_t\right\rbrace_{n=1}^{\infty}$ is uniformly  bounded in the $1-$Wasserstein space. It follows from the $C_0(l_0)-$Lipschitz continuity and uniform boundness of $\left\lbrace \sigma^n_t\right\rbrace_{n=1}^{\infty}$ that there exists a subsequence of $\left\lbrace \sigma^n\right\rbrace_{n=1}^{\infty}$ still denoted $\left\lbrace \sigma^n_t\right\rbrace _{n}$ ($n$ is independent of $t$) such that $\left\lbrace \sigma^n_t\right\rbrace_{n=1}^{\infty}$ converges narrowly to some $\sigma_t$. In light of \eqref{eq:L^r}, the Dunford Pettis  theorem ensures that $\sigma_t$ is absolutely continuous with respect to Lebesgue and the weak lower semicontinuity of the $L^r-$ norm establishes (a). We note that,  in view of \eqref{eq:lip W1}, (c) is guaranteed by the lower semicontinuity of the Wasserstein distance with respect to the narrow convergence.\\
As $\left\lbrace \sigma^n\right\rbrace_{n=1}^{\infty}$ converges narrowly to $\sigma_t$ ,  $ \left\lbrace V_t[\Psi_n]\right\rbrace_{n=1}^{\infty}$ converges to $V_t[\Psi] $ a.e and $ \left\lbrace V_t[\Psi_n]\right\rbrace_{n=1}^{\infty}$ is bounded. These, combined with the fact that $\left\lbrace \sigma^n_t\right\rbrace_{n=1}^{\infty}$ satisfies (a), yield that  $ V_t[\Psi_n]\sigma^n_t$ converges to $ V_t[\Psi]\sigma_t $ in the sense of distributions fot $t$ fixed by standard convergence results.\endproof

\begin{thm}
 Assume the conditions (A1), (A2), (B1), (B2) and (B3) hold. Let $l>0$, $l_0>0$ and $T>0$ such that
 $e^{4MT}(4l_0+1)<l+1$. Let $c$ be as defined in \eqref{eq: cost function}. Let $\sigma_{0}\in\mathscr{P}^{ac}\left( \mathbb{R}^2\right)$ with $\spt(\sigma_{0})\subset B^+_{l_0}$. Let $\Psi_0\in C(B_{l_0}) $ with and $\varrho_0$ a monotone function such that  $(\Psi_0^c)_c= \Psi_0$ and $((\Psi_0^c),\Psi_0, \varrho_0)$ solves \eqref{eq: equiv problem}. Then, there exist $\left\lbrace \sigma_t\right\rbrace_{t\in[0,\;T]} \subset\mathscr{P}^{ac}\left( \mathbb{R}^2\right)$ with $\spt (\sigma_{t})\subset B_{l}$, $\left\lbrace \Psi_t\right\rbrace_{t\in(0,\;T]} \subset C\left( \mathbb{R}^2\right)$ with  $(\Psi_t^c)_c= \Psi_t$ and a sequence of monotone functions $\left\lbrace \varrho_t\right\rbrace_{t\in(0,\;T]}$ such that $((\Psi_t^c),\Psi_t, \varrho_t)$ solves \eqref{eq: equiv problem} for each $t\in[0,T]$. Moreover, $t\longmapsto\sigma_t$ is Lipschitz continuous on $[0,\;T]$, belongs to $AC_{1}\left( 0, T ;\mathscr{P}( \mathbb{R}^2)\right)$  and satisfies 
\begin{equation}\label{eq:cont.Eq1} 
\begin{cases}
 \frac{\partial\sigma}{\partial t}+ \div (\sigma V_t[\Psi_t])=0, \qquad  \mathcal{D}'\left( (0, T)\times\mathbb{R}^2 \right) \\
\sigma_{|t=0}=\bar\sigma_{0}.
\end{cases}
\end{equation}
\end{thm}

\proof{} Let $N$ be a positive integer. We divide the interval $[0,T]$ into $N$ sub-intervals, each of length $\tau=\frac{T}{N}$. We consider $\sigma_t^N$ on $[0,T]$ defined as follows: $\sigma^{N}_{|t=0}=\bar\sigma_0$ and $\sigma^{N}$ solves 
\eqref{eq:cont.Eq0} on $[0,\tau)$ for $t_0=0$ and $\Psi=\Psi_0$ thanks to lemma \ref{lem: Cont.Eq time step}. To construct $\sigma_t^N$ on  $[\tau,2\tau)$ we first choose $\Psi_{\tau}$ and $\varrho_{\tau}$  such that $(\Psi_{\tau}^c,\Psi_{\tau},\varrho_{\tau})$ solves \eqref{eq: equiv problem} when $\sigma$ is replaced by $\bar\sigma_0$. Then, $\sigma_t^N$ is obtained on  $[\tau,2\tau)$ as a solution of \eqref{eq:cont.Eq0} for $t_0=\tau$ and $\Psi=\Psi_{\tau}$  thanks to lemma \ref{lem: Cont.Eq time step}. We repeat this process $(N-2)$ more times on the intervals $[k\tau,(k+1)\tau)$, $2\leq k\leq N-1$ by choosing $(\Psi_{k\tau}^c,\Psi_{k\tau},\varrho_{k\tau})$ as a solution to \eqref{eq: equiv problem} when $\sigma$ is replaced by $\sigma_{k\tau}$. We point out that proposition \ref{change of variables}  guarantees the existence of  $(\Psi_{k\tau},\varrho_{k\tau})$  provided that the support of $\sigma_{k\tau}$ is bounded. We also point out that lemma \ref{lem: Cont.Eq time step} ensures that $\left\lbrace \sigma_{k\tau}^N\right\rbrace_{k=1}^{N}\subset \mathscr{P}^{ac}\left(  \mathbb{R}^2)\right)$ with $\spt(\sigma_{k\tau})\subset B_{l_{k\tau}}$, where $l_{(k+1)\tau}\leq l_{k\tau}+ C(l_{k\tau})\tau $ for $0\leq k\leq N-1$. In light of this construction,
$\sigma_t^N$  satisfies
\begin{equation}\label{eq: delayed cont.ed} 
\begin{cases}
 \frac{\partial\sigma^{N}_t}{\partial t}+ \div ({\sigma^{N}_t\bf v}^N_t)=0, \qquad  \calD'\left( (0,T)\times\mathbb{R}^2 \right) \\
\sigma^N_{|t=0}=\bar\sigma_0.
\end{cases}
\end{equation}
Here, ${\bf v}^N_t= V_t[\Psi_{k\tau}]$ for $k\tau\leq t<(k+1)\tau$. We next show that the support of $\sigma^N$ is uniformly bounded independently of $N$ provided that  $e^{4MT}(4l_0+1)<l+1$. To that aim, we write $l_k=l_{k\tau}$ for simplicity. We thus have 
\begin{equation}\label{eq: }
l_{k+1}\leq l_{k}+C_0(l_k)\tau\leq l_{k}+M(4l_k+1)\tau= (4M\tau+1)l_{k} +M\tau. 
\end{equation}
By an inductive argument, we easily show that 
$$4l_{k+1}\leq (4M\tau+1)^k\left(4l_0+1 \right)-1. $$
For $k=N$, we get 
$$4l_{N+1}\leq \left( \frac{4MT}{N}+1\right)^N\left(4l_0+ 1 \right)-1\leq e^{4MT}\left(4l_0+1 \right)-1<l. $$

In light of lemma \ref{lem: Cont.Eq time step} (a), the construction above yields 
\begin{equation}\label{eq: Lr estimate}
\int_{\mathbb{R}^2}\left(  \frac{\partial \sigma^N_t}{\partial \calL^2}\right)^r d\q\leq \int_{\mathbb{R}^2} \left(  \frac{\partial \bar\sigma_{0}}{\partial \calL^2}\right)^{r}d\q
\end{equation}
  for any  $ r\geq 1 $ and $t\in [0, T]$. We use lemma \ref{lem: Cont.Eq time step}(c) to obtain that $\sigma_{N}$ is $C_0(l)$-Lipschitz continuous on $[0,T]$. Since  $\sigma^N$ satisfies (iii) and $\sigma^{N}_{|t=0}=\bar\sigma_0$ for all $N>0$, in light of standard compactness results, we assume without loss of generality that $\left\lbrace \sigma_t^N\right\rbrace_{N=1}^{\infty}$ converges narrowly to 
some $\sigma_t\in \mathscr{P} \left( \mathbb{R}^2\right)$. As $\sigma_t^N$ satisfies \eqref{eq: Lr estimate}, the Dunford-Pettis theorem yields that $\sigma_t\in \mathscr{P}^{ac} \left( \mathbb{R}^2\right)$. The narrow convergence of $\left\lbrace \sigma_t^N\right\rbrace_{N=1}^{\infty}$ combined with  weak semi-continuity of the $W_1$ leads to  $C_0(l)$-Lipschitz continuous on $[0,T]$. Define $\bar\sigma_t^N$  by $\bar\sigma_t^N:=\sigma_{k\tau}$ for $t\in [k\tau, (k+1)\tau)$, $0\leq k\leq N-1$.
We note that 
$$ W_1\left(\sigma^{N}_t,\bar{\sigma}^{N}_t \right)=  W_1\left(\sigma^{N}_t,\sigma_{\tau} \right)\leq | t-k\tau| \leq \frac{T}{N}, \qquad \text{for} \; t\in [k\tau, (k+1)\tau). $$ 
It follows that, as $\left\lbrace \sigma_t^N\right\rbrace_{N=0}^{\infty}$ converges narrowly to $\sigma_t$, $\left\lbrace \bar\sigma_t^N\right\rbrace_{N=1}^{\infty}$ converges narrowly to $\sigma_t$ for $t\in[0,T]$. Consequently, $\left\lbrace {\bf v}_t^N\right\rbrace_{N=0}^{\infty}$ converges to $V_t[\Psi]\quad \calL^1-a.e\; t\in [0,T]$,  by using proposition \ref{prop: stability transport}. As $\left\lbrace {\bf v}^N_t\right\rbrace_{N=1}^{\infty}$ is uniformly bounded in the $L^{\infty}(\mathbb{R}^2)$ and $\left\lbrace \sigma_t^N\right\rbrace_{N=0}^{\infty}$ satisfies \eqref{eq: Lr estimate}, we have that $\left\lbrace {\bf v}^N_t\sigma_t^N\right\rbrace_{N=1}^{\infty}$ converges in the sense of distributions to $V_t[\Psi_t]\sigma_t$ for a.e $t\in [0,T]$. Thus, we have that $t\rightarrow\sigma_t$ solves \eqref{eq:cont.Eq1}. \endproof

%
%
%
%
%

\section*{Acknowledgments}
Marc Sedjro is supported by the Alexander Von Humblodt Foundation. The project on which this publication is based has been carried out with funding from
the German Federal Ministry of Education and Research under project reference
number 01DG15010. 
%



\end{document}